\documentclass{birkjour}
\usepackage{amssymb,amsthm,amsfonts,amsmath}

\newcommand{\cS}{{\mathcal{S}}}

\newcommand{\cN}{{\mathcal{N}}}
\newcommand{\cC}{{\mathcal{C}}}
\newcommand{\cD}{{\mathcal{D}}}

\newcommand{\Hh}{{\mathcal{H}}}
\newcommand{\cG}{{\mathcal{G}}}
\newcommand{\cK}{{\mathcal{K}}}

\newcommand{\cR}{{\mathcal{R}}}

\newcommand{\E}{{E_{V_T^2}}}

\newcommand{\ns}{{\mathrm{nsa}}}
\newcommand{\sa}{{\mathrm{sa}}}
\newcommand{\anti}{{\mathrm{anti}}}
\newcommand{\Imt}{{\mathrm{Im}}}
\newcommand{\sign }{{\mathrm{sign}}}
\newcommand{\nein}[1]{}

\newtheorem{tht}{Theorem}
\newtheorem{thd}{Definition}
\newtheorem{thl}[tht]{Lemma}
\newtheorem{thp}[tht]{Proposition}

\theoremstyle{plain}

\newtheorem{thex}{Example}
\newtheorem{rem}{Remark}
\newcommand{\ov}{\overline}

\newcommand{\dT}{\mathbb{T}}

\newcommand{\dR}{\mathbb{R}}
\newcommand{\dN}{\mathbb{N}}
\newcommand{\dC}{\mathbb{C}}

\newcommand{\io}{\iota}

\begin{document}

\begin{title}
{Antilinear Normal Operators  on Hilbert Space}
\end{title}
\author{Konrad Schm\"udgen}
\address{University of Leipzig, Mathematical Institute, Augustusplatz 10, D-04109 Leipzig, Germany}
\email{schmuedgen@math.uni-leipzig.de}

\begin{abstract}
An operator $A$ on a complex Hilbert space $\Hh$ is called antilinear if $A(x+y)=Ax+Ay$ and $A(\lambda x)=\ov{\lambda} Ax$ for $x,y\in \cD(A)$ and $\lambda\in \dC$. We investigate some classes of densely defined antilinear unbounded operators, especially antilinear normal operators. We give various characterizations of antilinear normal operators and study a class of such operators in detail. Our main result is a structure theorem for  unbounded antilinear normal  operators.
\end{abstract}
\maketitle


\textbf{Key  words:} conjugation, antilinear operator, antilinear normal operator

\section {Introduction}

The theory of bounded or unbounded linear operators on Hilbert space is well developed and there exists a huge quantity of research papers and books on linear  operators.  In constrast,  conjugate-linear operators, or equivalently, antilinear operators (that is,  operators $T$ satisfying $$T(\alpha x+\beta y)=\ov{\alpha}\, T(x)+\ov{\beta}\, T(y) \quad \textrm{for} \quad\alpha, \beta\in \dC, \, x, y\in \cD(T))$$ are rarely investigated in the literature. 

Antilinear operators occur in the study of unbounded $C$-symmetric operators and  $C$-self-adjoint operators.
A  densely defined linear operator $T$ on a Hilbert space $\Hh$ is called $C$-symmetric (resp. $C$-self-adjoint) with respect to a conjugation $C$ on $\Hh$ if $CTC\subseteq T^*$ (resp.  $CTC= T^*).$ For such operators, the {\it antilinear} operators $TC$ and $T^*C$ play an important role (see \cite{knowles}, \cite{race}, \cite{garcia}, \cite{arlinski}). In the finite dimensional case,  $C$-self-adjoint  operators can be described by  complex symmetric matrices, a  well-studied class of matrices   (see e.g. \cite{fassbender},  \cite[Chapter 11]{gantmacher}).

Antilinear unitary operators appear in quantum physics since the work of E. Wigner. By Wigner's theorem \cite{wigner31}, each symmetry of the set of  pure states of a quantum  mechanical system is given  by a unitary or by an antiunitary operator. Wigner \cite{wigner} also classified the antiunitary operators $A$ for which the unitary operator $A^2$ has a complete set of eigenvectors.  Further, the $\mathcal{P}\mathcal{T}$- symmetry operator in quantum mechanics is an antilinear operator (see e.g. \cite{sw}).

 The subject of  the present paper are densely defined closed antilinear unbounded operators on general complex Hilbert spaces. Our particular emphasize is on the structure of unbounded antilinear normal operators.

Let us  describe the contents of this paper. In Section
\ref{antilinearoperators}, we develop basic definitions on antilinear operators. For a densely defined antilinear operator $T$  its adjoint $T^\dagger$ is defined by 
$$
\langle Tx,y\rangle= \langle T^\dagger y, x\rangle \quad \textrm{for}\quad x\in \cD(T), y\in \cD(T^\dagger ).
$$
This is  Wigner's classical definition \cite{wigner}. We show that for each densely defined closed antilinear operator $T$ the linear operator $T^\dagger T$ is self-adjoint and  positive, so the definition\, $|T|:=\sqrt{T^\dagger T}$\, of the absolute value of $T$ makes sense. 

In Section \ref{polarclasses},  the  polar decompositon of densely defined closed antilinear  operators is obtained (Theorem \ref{polarde}) and we introduce some natural classes of antilinear operators.
 As in  the linear  case, we define that $T$ is self-adjoint if $T=T^\dagger$ and $T$ is normal if $T^\dagger T=TT^\dagger $. Antilinear normal operators are closely related to $C$-normal operators as defined in \cite{ptak}, see Remark \ref{Cnormal}. A densely defined closed antilinear operator $T$  if self-adjoint if and only if $T=C|T|=|T| C$ for some conjugation $C$ (Theorem \ref{charselfadjoint}).

In Section \ref{gennormal}, we give some basic characterizations of antilinear normal operators. It is shown that an antilinear operator $T$ is normal if and only if  there exists a conjugation $C$ such the product $TC$ is a linear normal operator (Proposition \ref{tnormalconjugation1} and Theorem \ref{tnormalconjugation}).  Further, a densely defined closed antilinear operator $T$ is normal if and only if there exists an antilinear unitary operator $V$ such that $T=V |T|=|T|V$ and $T^\dagger =V^\dagger |T|=|T|V^\dagger $ (Theorem \ref{charnpolar}).

 Section \ref{2model} is devoted to 
 antilinear normal operators on the two-dimensional Hilbert space $\dC^2$. This simple case illustrates the general case and it delimits possible results.

In Section \ref{nmodel}, we study an important model of antilinear normals. Let $N_1$ and $N_2$ be normal linear operators on Hilbert spaces $\Hh_1$ and $\Hh_2$, respectively, and define a normal linear operator $\cN$ on  $\Hh:=\Hh_1\oplus \Hh_2$ by the operator matrix \begin{gather}
\cN=\left(
\begin{array}{ll}
N_1 & ~ 0 \\
~0 &  N_2
\end{array}\right).
\end{gather}
Further, 
let  $\cC$ be a conjugation of the Hilbert space $\Hh_1\oplus \Hh_2$ of the form
\begin{gather} 
\cC=\left(
\begin{array}{ll}
 ~0 & C_2\\
C_1 &  ~0 
\end{array}\right)
\end{gather}
such that $\cN\cC=\cC\cN$. Then the operator matrix
\begin{gather} 
T:=\cN\, \cC=\left(
\begin{array}{ll}
~~~ 0 & N_1C_2 \\
N_2C_1 &  ~~~ 0
\end{array}\right).
\end{gather} 
defines an antilinear normal operator $T$ on $\Hh=\Hh_1\oplus \Hh_2$. Such  operators $T$ are  investigated in Section \ref{nmodel} in detail. 

The main result of this paper is Theorem \ref{normalform} which describes the structure of antilinear normal operators. It is stated and discussed in  Section \ref{structure} and proved in Section \ref{proof}. According to this result  each antilinear normal operator $T$ decomposes as an orthogonal direct sum of an antilinear self-adjoint operator $T_\sa$ and an antilinear completely non-self-adjoint operator $T_\ns$. The non-self-adjoint part $T_\ns$ is of the form of  operators  treated in Section \ref{nmodel}. In particular, the results obtained in Section \ref{nmodel} provide criteria for the unitary equivalence and the irreducibility of such operators $T_\ns$. 

In the  finite-dimensional case, antilinear normal operators were studied and their structure was classified by  F. Herbut and M. Vujicic \cite{herbut} adapting Wigner's description of antiunitaries. As note above, 
the present paper is about antilinear operators on general Hilbert spaces including  in particular unbounded operators.
\smallskip

Let us fix some notational conventions that will be used throughout this paper.
All Hilbert spaces  are complex and their scalar products are linear in the first variables and conjugate-linear in the second. The symbols $\Hh$, $\Hh_1$, $\Hh_2$ always stand for Hilbert  spaces. For an operator $T$, its domain, its kernel and its range are denoted by $\cD(T)$, $\cN(T)$ and $\cR(T)$, respectively.
The spectral measure  of a normal or self-adjoint linear operator $N$ is denoted by $E_N(\cdot)$.

For an antilinear operator $A$ on $\Hh$ we define  the resolvent set $\rho(A)$ to be the set of numbers $\lambda$ for which the operator $A-\lambda I$ has a bounded inverse $(A-\lambda I)^{-1}$ defined on $\Hh$. Note that $(A-\lambda I)^{-1}$ is then a {\it real-linear} operator (\cite{huhtanen}), that is, 
$$
(A-\lambda I)^{-1}(\alpha x+\beta y)=\alpha\, (A-\lambda I)^{-1}x +\beta\, (A-\lambda I)^{-1}y \quad \textrm{for} ~\alpha, \beta\in \dR, \,x, y\in \Hh.$$ 
As in the linear case, the spectrum $\sigma(A)$ is the complement set $\dC\backslash \rho(A).$

For a subset $M$ of $\dC$ let $\ov{M}$ denote the closure of $M$ and $M^c:=\{\ov{z}: z\in M\}$  the complex conjugate of $M$.

\section{Antilinear  Operators}\label{antilinearoperators}

\begin{thd}
Let $T$ be mapping of a Hilbert space $\Hh_1$ into a Hilbert space $\Hh_2$ defined on a linear subspace $\cD(T)$ of $\Hh_1$,  the \emph{domain} of $T$. Then $T$  is called a \emph{conjugate-linear operator}, or an \emph{anilinear operator}, if
\begin{align*}
T(\alpha x+\beta y)=\ov{\alpha}\, T(x)+\ov{\beta}\, T(y)\quad \textrm{for} ~~~\alpha, \beta\in \dC,~ x,y\in \cD(T).
\end{align*} 
\end{thd}
In this paper, we prefer to speak about {\it antilinear} operators rather than {\it conjugate-linear} operators.

Let $\Hh$ be a Hilbert space with scalar product $\langle\cdot,\cdot\rangle_\Hh$. We define another Hilbert space $\Hh^\anti$ as follows. The sets of $\Hh$ and $\Hh^\anti$ are the same  and also the additions of vectors. The complex multiple $\lambda \circ x$ of $x\in\Hh$ in $\Hh^\anti$ is equal to the multiple $\ov{\lambda} \, x$ in $\Hh$ and the scalar product of $\Hh^\anti $ is given by 
\begin{align*}\langle x,y\rangle_{\Hh^\anti}=\langle y,x\rangle_\Hh,\quad x,y \in \Hh.
\end{align*}
To avoid confusion, for $\lambda \in\dC$ and $x\in \Hh$, the symbol   $\lambda \circ x$ means the complex multiple of $x$ in $\Hh^\anti$ and $\lambda  x$ denotes the complex multiple
 of $x$ in $\Hh$.  

We verify (for instance) linearity and conjugate-linearity of the new scalar product $ \langle \cdot,\cdot \rangle_{\Hh^\anti}$: For $\lambda, \mu\in \dC$ and $x,y\in \Hh$, we derive
\begin{align*}
\langle \lambda \circ x,\mu\circ y\rangle_{\Hh^\anti}=\langle \ov{\lambda}\, x,\ov{\mu}\, y\rangle_{\Hh^\anti}=\langle \ov{\mu}\, y,\ov{\lambda}\, x\rangle_\Hh =\ov{\mu}\, \lambda\,   \langle y,  x\rangle_\Hh =\lambda\, \ov{\mu}
  \langle x,y\rangle_{\Hh^\anti}.
\end{align*}
Note that $(\Hh^\anti)^\anti$ is just the Hilbert space $\Hh$.
The embedding of $x\in \Hh$ into $ \Hh^\anti$ is denoted by $\iota(x)$. In particular, for  $y\in \Hh^\anti$, $\iota(y)$  denotes the embedding of $y$ into $\Hh=(\Hh^\anti)^\anti$.

Let $T$ be a mapping of a Hilbert space $\Hh_1$ into a Hilbert space $\Hh_2$. Clearly, $T:\Hh_1\to \Hh_2$ is an antilinear operator if and only if $T:\Hh_1^\anti\to \Hh_2$ is a linear  operator, or equivalently, $T:\Hh_1\to \Hh_2^\anti$ is a linear operator. This enables us to  study  antilinear operators by means of linear operators.

For  mappings $T:\Hh_1\to \Hh_2$ and $S:\Hh^\anti_1\to \Hh_2$ we define  mappings $j(T):\Hh_1^\anti\to \Hh_2$ and $j(S):\Hh_1\to \Hh_2$  by
\begin{align*}
j(T)(y):&= T(\io(y)), \quad y\in \Hh_1^\anti,\\ j(S)(x):&= S(\iota(x)), \quad x\in\Hh_1.
\end{align*}
Notions such as closedness, boundedness or density of domains are the same for $T$ and $j(T)$ and for $S$ and $j(S)$, because $\Hh_1$ and $\Hh_1^\anti$ have the same norms.

Let us recall the definition of the adjoint of a densely defined linear operator $S:\Hh_1\to \Hh_2$. The domain of $S^*$ is
$$
\cD(S^\ast){=}\{ y \in \Hh_2:\text{There exists $u \in \Hh_1$ such that $\langle Sx, y \rangle_2 {=} \langle x,u \rangle_1$ for $ x \in \cD(S)$} \}.
$$
Since $\cD(S)$ is dense in $\Hh_1$, then the vector $u \in \Hh_1$  is uniquely determined by $y$. Setting  $S^\ast y= u$, $S^*$ becomes   a  linear operator $S^\ast:\Hh_2\to \Hh_1$ such that
\begin{align}\label{tstar}
\langle Sx,y \rangle_2 = \langle x, S^\ast y \rangle_1 \quad\text{for}~~  x \in \cD(S),~~ y \in \cD(S^\ast).
\end{align}

Now suppose that $T$ is a densely defined {\it antilinear} operator of $\Hh_1$ into $\Hh_2$. Then $j(T):\Hh_1^\anti\to\Hh_2$ is a {\it linear} operator with dense domain in $\Hh_1^\anti$,  so  its  adjoint $j(T)^*:\Hh_2\to \Hh_1^\anti$ is a linear operator which is defined by the preceding. We define the adjoint $T^\dagger$ of $T$ by
$$
T^\dagger y:=\io(j(T)^*y), \quad y\in \Hh_2.
$$
This means that
\begin{align*}
\cD(T^\dagger){=}\{ y \in \Hh_2:&\text{There exists  $u \in \Hh_1$ such that $\langle Tx, y \rangle_2 {=} \langle u,x \rangle_1$ for $ x \in \cD(T)$} \}, \\ & \quad\quad\quad\quad \quad T^\dagger y=u, \quad y\in \cD(T^\dagger),
\end{align*}
Since $j(T)^*:  \Hh_2\to \Hh_1^\anti$ is linear,  $T^\dagger=\io\circ j(T)^*:\Hh_2\to \Hh_1$ is an antilinear operator and equation (\ref{tstar}) reads as
\begin{align}\label{tstar1}
\langle Tx,y \rangle_2 = \langle T^\dagger y,x \rangle_1 \quad\text{for}~~  x \in \cD(T),~~ y \in \cD(T^\dagger).
\end{align}
\begin{thd}
The antilinear operator $T^\dagger$ defined above is called the {\em adjoint operator} of the antilinear operator $T$.
\end{thd}
Note that (\ref{tstar1}) coincides with the definition of the adjoint  of an antilinear operator according to Wigner \cite{wigner}. The preceding approach to the definition of $T^\dagger$ can be considered as a mathematical justification of Wigner's definition.

As usual, the adjoint of a linear operator $S$ is denoted by $S^*$.
Note that for an antilinear operator $T$ and a linear operator $S$ we have $(cT)^\dagger= c\, T^\dagger$ and $(cS)^*=\ov{c}\, S^*$ for $c\in \dC$, $c\neq 0.$

If  $A$ is a linear operator and $S$ and $T$ antilinear operators, then $AT$, $TA$ are antilinear,  $ST$ is linear and 
\begin{align*}
T^\dagger A^*\subseteq (AT)^\dagger,~~~ A^*T^\dagger \subseteq (TA)^\dagger,~~~ S^\dagger T^\dagger \subseteq (TS)^*.
\end{align*}

Let $A:\Hh_1\to \Hh_2$ be a densely defined antilinear operator. Then
\begin{align*}
j(A)^*j(A):\Hh_1^\anti\to \Hh_1^\anti, \quad
j(A)j(A)^*:\Hh_2\to \Hh_2.
\end{align*}
From the corresponding definitions we obtain
\begin{align}\label{adaggera}
A^\dagger A=\io\circ (j(A)^* j(A))\circ\io
\end{align}
\begin{thp}\label{adag}
Suppose that $A$ is a densely defined closed antilinear operator of $\Hh_1$ into $\Hh_2$. Then $A^\dagger A$ and $AA^\dagger $ are  positive self-adjoint operators on $\Hh_1$ and $\Hh_2$, respectively. Moreover, $\cD(A^\dagger A)$ is a core for the operator $A$.
\end{thp}
\begin{proof} 

We carry out the proof for $A^\dagger A$; the case of $AA^\dagger $ is similar. 

For $x,y \in \cD(A^\dagger A)$, using (\ref{tstar1}) twice we compute
$$\langle A^\dagger Ax,y\rangle=\langle Ay, Ax\rangle =\langle x, A^\dagger Ay\rangle.$$
Hence $A^\dagger A$ is symmetric. Since $\langle A^\dagger A x,x\rangle=\langle Ax,Ax\rangle\geq 0$, $A^\dagger A$ is positive.

Since $\|Ax\|=\|j (A)\io(x)\|$ for $x\in \cD(A)$ and $A$ is closed  by assumption, the linear operator $j(A)$ of $\Hh_1^\anti$ into $\Hh_2$ is closed. Hence $j(A)^*j(A)$ is positve self-adjoint linear operator  on the Hilbert space $\Hh_1^\anti$ (see e.g. \cite[Proposition 3.18(ii)]{schm12}). In particular it follows therefore from (\ref{adaggera}) that $A^\dagger A$ is closed.

Since $A^\dagger A$ is a closed positive symmetric operator, to prove its self-adjointness it suffices to show that $\cR(A^\dagger A+\lambda I)$ is dense in $\Hh_1$ for $\lambda>0$. Assume that $y\in \Hh_1$ and $y\perp \cR(A^\dagger A+\lambda I)$. Then, for $x\in \cD(A^\dagger A)$,
\begin{align*}
0=&\langle y, (A^\dagger A +\lambda)x\rangle_{\Hh_1}= \langle y, \io\circ j(A)^*j(A)\io(x)\rangle_{\Hh_1}+ \langle y, \lambda x\rangle_{\Hh_1}\\ &=
 \langle j(A)^*j(A)\io(x), \io(y)\rangle_{\Hh_1^\anti}+\langle \io(\lambda x),\io(y)\rangle_{\Hh_1^\anti}\\ &=
\langle ( j(A)^*j(A)+\lambda)\io(x), \io(y)\rangle_{\Hh_1^\anti}.
\end{align*}
As noted above, $j(A)^*j(A)$ is positive self-adjoint operator on $\Hh_1^\anti$. Therefore, the preceding implies that $\io(y)=0$. Thus $y=0$. Hence  $\cR(A^\dagger A+\lambda I)$ is dense and $A^\dagger A$ is self-adjoint.

For the  linear operator $j(A):\Hh^\anti_1\to \Hh_2$ it is well-known that $\cD(j(A)^*  j(A))$ is a core for $j(A)$, see e.e. \cite[Proposition 3.18]{schm12}. Since the norms of $\Hh_1^\anti$ and $\Hh_1$ are the same, it follows from (\ref{adaggera}) that $\cD(A^\dagger A)$ is a core for  $A$.
\end{proof}

One could  try to study antilinear operators by passing to linear operators on the Hilbert space $\Hh^\anti$ as done in the   proof of Proposition \ref{adag}.  This has the disadvantage that an antilinear operator $T:\Hh\to \Hh$ becomes a linear operator of $\Hh$ into another Hilbert space $\Hh^\anti$.
We will use another way and try to express an antilinear operator as a  product of some appropriate linear operator and some conjugation.

\section{Polar Decomposition and Classes of Operators}\label{polarclasses}
For a densely defined closed linear operator $T$ its absolute value  $|T|$ is the unique positive square root $\sqrt{T^*T}$ of the positive self-adjoint operator $T^*T$.
\begin{thd}
Let $A$ be a densely defined closed antilinear operator  from $\Hh_1$ into $\Hh_2$.  The \emph{absolute value} $|A|$ of $A$ is the positive square root of the linear positive self-adjoint operator  $A^\dagger A$ (by Proposition \ref{adag}):
\begin{align*} 
|A|:=\sqrt{A^\dagger A}.
\end{align*}
\end{thd}

As in the linear case we have the following fact.
\begin{thl}Suppose  $A$ is a  densely defined closed antilinear operator. Then
\begin{align}\label{danorm}
  \cD(A)=\cD(|A|)\quad \text{and} \quad \|Ax\|=\| \,|A|\, x\| ~~~ \text{for}~~~ x\in \cD(A).
\end{align}
\end{thl}
\begin{proof}
The proofs follows verbatim the same pattern as in the linear case (\cite[Lemma 7.1]{schm12}). For
$x\in \cD(A^\dagger A)= \cD(|A|^2)$, we have
\begin{align}\label{aanorms}
\|A x\|^2=\langle Ax,Ax\rangle=\langle A^\dagger Ax,x\rangle=\langle |A|^2x,x\rangle= \|\, |A|\, x\|^2.
\end{align}
By Proposition \ref{adag}, $\cD(A^\dagger A)$ is a core for $A$. Also, $\cD(|A|^2)$ is core for the self-adjoint operator $|A|$. Equation (\ref{aanorms}) shows that the graph norms of $A$ and $|A|$ coincide on $\cD(A^\dagger A)= \cD(|A|^2)$. Taking the closures in (\ref{aanorms}) with respect to the graph norms gives (\ref{danorm}).
\end{proof}
\begin{thd}
An antilinear bounded operator $U$ defined on a Hilbert space $\Hh$ is called an \emph{antilinear partial isometry} if\,  $\|Ux\|=\|x\|$ for $x\in \cN(U)^\bot$.
\end{thd}
That is, $U$ is an antilinear isometric map of  the closed subspace $\cK_1:=\cN(U)^\bot$,  the initial space of $U$, on the closed subspace $\cR(U)$, the final space of $U$, and $U$ is zero on the complement of $\cK_1$.

The following formula (\ref{polarformula})  is  the polar decomposition for antilinear operators.
\begin{tht}\label{polarde} Suppose  $A:\Hh_1\to \Hh_2$ is a densely defined closed antilinear operator. Then  there exists a unique antilinear partial isometry $U_A$ with initial space $\ov{\cR(|A|)}$  and final space $\ov{\cR(A)}$, called the \emph{phase operator} of $A$, such that
\begin{align}\label{polarformula}
A=U_A|A|.
\end{align}
Moreover,
\begin{align}\label{polarstar}(U_A)^\dagger = U_{A^\dagger}.
\end{align}
\end{tht}
\begin{proof}
The proof proceeds as in the linear case; for the reader's convenience we sketch  the main reasoning. From (\ref{aanorms}) it follows that $U_A(|A|x)=Ax$ for $x\in \cD(A)= \cD(|A|)$ defines an isometric map of  $\cR(|A|)$  on $\cR(A)$. Since $A$ is antilinear and $|A|$ is linear, $U_A$ is antilinear. By continuity, $U_A$ extends to an isometric map of $\ov{\cR(|A|)}$  on $\ov{\cR(A)}$. Setting $U_Ay=0$ for
$y\in \cR(|A|)^\bot$, $U_A$ becomes an antilinear partial isometry. Clearly, the antilinear partial isometry $U_A$ with initial space $\ov{\cR(|A|)}$  and final space $\ov{\cR(A)}$ is uniquely determined by  (\ref{polarformula}).

We prove (\ref{polarstar}).
Set $T:=U_A|A|U_A^\dagger$. Using that $U_A^\dagger U_Ax=x$ for $x\in \ov{\cR(|A|)}$ and 
$A^\dagger=(U_A|A|)^\dagger =|A| U_A^\dagger$, we derive 
$$T^2=(U_A|A|U_A^\dagger)(U_A|A|U_A^\dagger))= U_A|A|^2U_A^\dagger =AA^\dagger.$$
The linear operator $T$ is  antiunitarily equivalent (via $U_A$) to the positive self-adjoint operator $|A|$. Hence $T$ is also positive and self-adjoint, so $T=|A^\dagger|$ by the uniqueness of the positive square root. Then
$$
U_A^\dagger |A^\dagger|=U_A^\dagger\, T= U_A^\dagger\,  U_A|A| U_A^\dagger=|A|\, U_A^\dagger=(U_A|A|)^\dagger =A^\dagger=U_{A^\dagger} |A^\dagger|,
$$
so  $(U_A)^\dagger = U_{A^\dagger}$ by the uniqueness of the phase operator.
\end{proof}

Many standard notions for linear operators carry over almost verbatim to antilinear operators. It is very easy to ``guess" the following definitions.
\begin{thd}\label{operatorclasses}
A densely defined antilinear operator $T$ on $\Hh$ is called\\
$\bullet$~ \emph{symmetric} if $T\subseteq T^\dagger$,\\
$\bullet$~ \emph{self-adjoint} if $T= T^\dagger$,\\
$\bullet$~ \emph{formally normal} if  $T^\dagger T\subseteq TT^\dagger$,\\
$\bullet$~ \emph{normal} if $T^\dagger T=T T^\dagger $,\\
$\bullet$~  \emph{antiunitary} if $T$ is a bijection of $\Hh$ and $\|Tx\|=\|x\|$ for $x\in \Hh$.
\end {thd}

Note that an antilinear densely defined operator $T$ is formally normal if and only if  $\cD(T)\subseteq \cD(T^\dagger)$ and $\|Tx\|=\| T^\dagger x\|$ for all $x\in \cD(T)$; likewise $T$ is normal if and only if  $T$ is formally normal and $\cD(T)=\cD( T^\dagger)$.

The following theorem clarifies the structure of antilinear self-adjoint operators. For bounded operators, this result was stated as \cite[Proposition 3.1]{routs}. The unbounded case requires  additional arguments.
\begin{tht}\label{charselfadjoint}
A densely defined closed antilinear operator $T$ on $\Hh$ is self-adjoint if and only there is a conjugation $C$ on $\Hh$ such that  $T=C|T|=|T|C$. In this case, we have $CE_{|T|}(M)= E_{|T|}(M)C$ for each Borel set $M$ of $[0,+\infty)$.
\end{tht}
\begin{proof}
The if direction is easy. Assume that $T=C|T|=|T|C$ for some conjugation $C$. Then, since $C$ is bounded, $T^\dagger=(C|T|)^\dagger= |T|^*C=|T|C=T$, that is, $T$ is self-adjoint.

Now suppose that $T$ is self-adjoint. 
Let $\Hh_0:=\cN(T)$. By (\ref{danorm}), $\Hh_0=\cN(|T|)$. Further,  $\Hh_1:=\cN(T)^\bot= \ov{\cR(T^\dagger) }=\ov{\cR(T)}$, because $T=T^\dagger$. Then, by Theorem  \ref{polarde}, $U_T$ is an isometric map of $\Hh_1$ onto $\Hh_1$ and $U_T=0$ on $\Hh_0$. 

Since $U_T=U_T^\dagger$ by (\ref{polarstar}) we  have $U_T|T| =T=T^\dagger=|T| U_T$ and $U_T T=U_T^2 |T|=U_T |T|U_T=TU_T$. These formulas imply that the decomposition $\Hh=\Hh_0\oplus \Hh_1$ reduces the operators $T$ and $|T|$, so there exists an operator $T_1$ on $\Hh_1$ such that $T=0\oplus T_1$ and $|T|=0\oplus |T_1|$. Then $T_1$ is an antilinear self-adjoint operator on $\Hh_1$ with trivial kernel. The corresponding phase operator $U_{T_1}$ is antiunitary and self-adjoint (by (\ref{polarstar})), so it is a  conjugation on $\Hh_1$, and it satisfies  $T_1=U_{T_1} |T_1|=|T_1| U_{T_1}$. Let $U_0$ be an arbitrary conjugation on $\Hh_0$. Setting $C:=U_0\oplus U_{T_1}$, we have $T=C|T|=|T|C$. 

The last assertion follows at once from \cite[Proposition 5.15]{schm12}. This result is stated in \cite{schm12} for linear operators $S$, but the proof applies verbatim to the antilinear operator $C$ as well.
\end{proof}

The next proposition implies that the spectrum of antilinear self-adjoint operators is  not empty and invariant under multiplication by complex numbers of modulus one. 
\begin{thp}\label{irrselfadjoint}
Suppose that $T$ is an antilinear self-adjoint operator. A complex number $\lambda$ is in $\sigma (T)$ if and only if $|\lambda|\in \sigma(|T|).$
\end{thp}
\begin{proof} \cite[Proposition 2.15]{huhtanen}, where the result is  stated for bounded operators. But the proof goes verbatim through in the unbounded case.
\end{proof}

\begin{thp} For $t\geq 0$,  we define 
\begin{align}\label{defist}
S_t(z)=t\, \ov{z} \quad \textrm{for} \quad 
z\in \Hh:=\dC.
\end{align} Then  is an irreducible antilinear self-adjoint operator. Each  irreducible antilinear self-adjoint operator is unitarily equivalent to some operator $S_t$ with unique number $t\in [0,+\infty)$.
\end{thp}
\begin{proof}
Obviously,  $S_t$ is  irreducible and self-adjoint. We prove that last assertion. Suppose that $S$ is an irreducible antilinear self-adjoint operator.

First we show that the spectrum of $|S|$ consists of a single point. Assume to the contrary that $\sigma(|S|)$ contains at least two  points. 
Then there exists a Borel set $M$ such that $\Hh_1:=E_{|S|}(M)\Hh\neq \{0\}$ and $\Hh_2:=E_{|S|}(\dR\backslash M)\Hh\neq \{0\}$. By Proposition \ref{charselfadjoint}, $S$ is of the form $S=|S|C$ for some conjugation $C$ and $C$ commutes with the spectral projections of $|S|$. Hence $\Hh_1$ and $\Hh_2$ reduce the conjugation $C$ and so $S=|S|C$, which contradicts the irreducibility of $S$. 

Thus, $\sigma(|S|)=\{t\}$ for some $t\in [0,+\infty).$ Then $S=t_0 C$   by Proposition \ref{charselfadjoint}. Since we can find an orthonormal basis of vectors which are invariant under $C$, the irreducibility of $S$ implies that $\dim \Hh=1$, say $\Hh=\dC\cdot v$. Then $Cv=\alpha v$, with $|\alpha|=1$. If $\beta$ is a square root of $\alpha$,  $w:=\beta v$ satisfies $Cw=w$. Then  there is a unitary operator mapping of $\Hh$ on $\dC$  which maps $w$ to $1$ and gives the unitary equivalence of $S$ and $S_t.$

\end{proof}
\section{Antilinear Normal Operators}\label{gennormal}

As mentioned in the introduction, our main emphasize in this paper is on antilinear normal operators.

A simple characterization of normality  is given in the following lemma.
\begin{thl}\label{ttchara}
Let $T$ be a densely defined closed  antilinear operator. Then $T$ is normal if and only if $|T|=|T^\dagger |$. 
\end{thl}
\begin{proof}
If $T$ is normal, then $T^\dagger T=TT^\dagger\geq 0$.  Taking the positive square root of the positive self-adjoint operators (by Proposition \ref{adag}) on both sides of this equality we obtain  $|T|=|T^\dagger |$. Conversely, squaring $|T|=|T^\dagger |$ gives $T^\dagger T=TT^\dagger$.
\end{proof}
A general method for the study of  antilinear normals is to write them as products of  linear normal  operators with   conjugations.
\begin{thp}\label{tnormalconjugation1}
Let $N$ be a linear normal operator.  Then there is a conjugation $C$ on $\Hh$ such that $T:=NC$ is an antilinear normal operator and $T^\dagger T=N^*N$.
\end{thp}
\begin{proof}
Since the normal operator $N$ is closed,  $N^*N$ is a self-adjoint operator. By the spectral representation theorem, the operator $N^*N$ is up to unitary equivalence a multiplication operator by some real function on a measure space. Let $C$ denote  the standard conjugation on this measure space defined by $(Cf)(t)=\ov{f(t)}$. Then   $CN^*NC=N^*N$.
Clearly, $T=NC$ is a densely defined antilinear operator. Using that $N$ is normal we derive 
\begin{align}\label{tstartt1} T^\dagger T=(NC)^\dagger NC=CN^*NC=N^*N=NN^*=(NC)(NC)^\dagger  =TT^\dagger. 
\end{align}
Thus, $T$ is an antilinear normal operator.
\end{proof}

\begin{tht}\label{tnormalconjugation}
Suppose that $T$ is an antilinear normal operator. Then there exists a conjugation $C$ such $N:=TC$ is a linear normal operator and $T=NC$, \begin{align}\label{ntcond}
|N|=|T|\quad \textrm{and}~~~  C|T|C=|T|.
\end{align}
If $E_{|T|}$ denotes the spectral measure of $|T|$, then $
CE_{|T|}(M)C=E_{|T|}(M)$ for all  Borel sets $M$ of $[0,+\infty)$.
\end{tht}
\begin{proof} Again we use  the spectral representation theorem and represent the  self-adjoint operator $|T|$  as  multiplication operator by some real function. Then the standard conjugation $C$ on this function space satisfies $C|T|=|T|C$ and
$C|T|C=|T|.$ Then 
 $$CT^\dagger TC=C|T|^2C=(C|T|C) (C|T|C)= |T|^2=T^\dagger T.$$
 Therefore, setting $N:=TC$ and using that $T^\dagger T=TT^\dagger$ by Lemma \ref{ttchara} we compute
\begin{align}
N^*N &=(TC)^*(TC)=CT^\dagger T C=T^\dagger T = TT^\dagger\nonumber\\ &= (TC)CT^\dagger=(TC)(TC)^*=NN^*.
\end{align}\label{nstarnn11}
Thus, $N$ is linear normal operator. Moreover, $N^*N=T^\dagger T$ 
 implies $|N|=|T|$.

From  \cite[Proposition 5.15]{schm12} it follows that $C|T|=|T|C$ is equivalent 
to $CE_{|T|}(\cdot )=E_{|T|}(\cdot)C$ and so to $CE_{|T|}(\cdot )C=E_{|T|}(\cdot)$. 
\end{proof}

\begin{rem} The preceding proof shows that for each conjugation $C$ satisfying $C|T|C=|T|$ the assertion of Theorem \ref{tnormalconjugation} holds. Thus, $C$ is highly non-unique in general. For instance, if $T$ is antiunitary, then $|T|=I$, so each conjugation $C$ on $\Hh$ can be taken. An antilinear operator $T$ on $\Hh$ is antiunitary if and only if $T=UC$ for some linear unitary operator $U$ and some conjugation $C$ on $\Hh$.
\end{rem}

\begin{rem}\label{Cnormal} Antilinear normal operators are closely related to the $C$-normal operators studied in \cite{ptak} and \cite{ko}. If $C$ is a conjugation on $\Hh$, then a densely defined closed linear operator $N$ on $\Hh$ is called \emph{$C$-normal}\, if $$(NC)(NC)^\dagger=(NC)^\dagger NC, $$ or equivalently, if $NN^*=CN^*NC$, that is, if the antilinear operator $NC$ is normal according to Definition \ref{operatorclasses}. From
Proposition \ref{tnormalconjugation1} and Theorem \ref{tnormalconjugation} (more precisely, from  formulas (\ref{tstartt1}) and (\ref{nstarnn11})) it follows that an antilinear operator $T$ is normal if and only if $T=NC$ for some $C$-normal operator $N$.
\end{rem}

The following auxiliary lemma is used in the proof of Theorem \ref{charnpolar}.

\begin{thl}\label{polarU} For each normal  linear operator $N$ there exists a unitary operator $U$ such that 
\begin{align}\label{unpol}
N=U|N|=|N|U\quad \textrm{and}~~~ N^*=U^*|N|=|N|U^*.
\end{align}
If $N$ is self-adjoint,  $U$ can be chosen to be self-adjoint and unitary.
\end{thl}
\begin{proof}
This result can be derived by means of properties of the polar decomposition and the normality of $N$. However, we prefer  to use the spectral representation theorem $N=\int_\dC z \, dE(z)$ for the normal operator $N$. Set $\dC^\times := \dC\backslash \{0\}$.  The orthogonal decomposition $\Hh=E(\{0\})\Hh\oplus E(\dC^\times)\Hh$ reduces the normal operator $N$ and its adjoint. Hence  the unitary operator $U$ on $\Hh$ defined by
$$
U:=I_ {E(\{0\})\Hh}\oplus \int_{\dC^\times} z |z|^{-1}\, dE(z)
$$
has the desired properties. Clearly, if $N$ is self-adjoint, the spectral measure is supported on $\dR$ and $U$ becomes self-adjoint.
\end{proof}
\begin{tht}\label{charnpolar}
Suppose that $T$ is a  densely defined closed  antilinear  operator. Then $T$ is normal if and only if there exists an antiunitary operator $V$ such that 
\begin{align}\label{tvrelations}T=V\,  |T|=|T|\, V\quad \textrm{and}~~~ T^\dagger=V^\dagger|T|=|T|V^\dagger.
\end{align}
Morevoer, in this case, $T^2$ is a linear normal operator and 
\begin{align}\label{T^2}
T^2=V^2|T|^2=|T|^2V^2.
\end{align}
\end{tht}
\begin{proof}
The if direction is  easy: Since $V$ is antiunitary, it follows immediately  from (\ref{tvrelations}) that $T^\dagger T= |T|^2$ and $TT^\dagger =|T|^2$, so $T$ is normal. 

Suppose now that $T$ is normal. Then,   by Proposition \ref{tnormalconjugation} there  exist a normal operator $N$ and a conjugation $C$ such that $T=NC$ and (\ref{ntcond}) holds. For the normal operator $N$ we choose a unitary $U$ as in Lemma \ref{polarU}.
Set $V:=UC$. Clearly, $V$ is an antiunitary. Using (\ref{unpol}) and  (\ref{ntcond}) we derive
\begin{align*}
T&=NC=U|N|C=UC|N|=V|T|,
 \\~~T&=U|N|C= |N|UC =|T|V,
\\~~ T^\dagger & =(NC)^\dagger=CN^*=CU^*|N|=V^*|T|,
\\~~ T^\dagger &=CU^*|N|=C|N|U^*=|N|CU^*=|T|V^*,
\end{align*}
which proves (\ref{tvrelations}).

In a similar manner one verifies (\ref{T^2}) and $(T^*)^2=|T|^2(V^2)^*=(V^2)^*|T|^2$.  Since $V$ is an antiunitary operator,  $V^2$ is a unitary with inverse $(V^2)^*$. Using these relations it is straightforward to verify that $T^2$ is normal.
\end{proof}

\section{Antilinear Normal Operators on $2$-dimensional Hilbert Space}\label{2model}

For $z\in \dC$ we define an antilinear mapping $T_z$ of $\dC^2$ by
\begin{align}\label{defTz}
T_z(x_1,x_2)=(\ov{z}\, \ov{x_2}, z\, \ov{x_1}), \quad (x_1,x_2)\in \dC^2.
\end{align}
Let $C$ denote the complex conjugation on $\dC$.
Then $T_z$ can be written as 
\begin{gather*} 
T_z=\left(
\begin{array}{ll}
\ov{z} & 0\\
~0 & z
\end{array}\right) \left(
\begin{array}{ll}
0 & C\\
C & 0
\end{array}\right)=\left(
\begin{array}{ll}
0 & C\\
C & 0
\end{array}\right) \left(
\begin{array}{ll}
\ov{z} & 0\\
~ 0& z
\end{array}\right)=\left(
\begin{array}{ll}
0 &\ov{z}\\
z & 0
\end{array}\right) \left(
\begin{array}{ll}
C & 0\\
0 & C
\end{array}\right).
\end{gather*}
One easily verifies that 
\begin{align}\label{Tz2}
T_z^\dagger=T_{\ov{z}}~~\textrm{and}~~~ 
T_vT_z=\left(
\begin{array}{ll}
\ov{vz}& ~0 \\
~~ 0 & vz
\end{array}\right) ~~~ \textrm{for}~~ v,z\in \dC.
\end{align}
In particular, 
\begin{align}\label{TdaggerT}
T_z^\dagger T_z=T_zT_z^\dagger =\ov{z}\, z\left(
\begin{array}{ll}
1& 0 \\
 0 & 1
\end{array}\right), ~~~ z\in \dC.
\end{align}
Thus, for any $z\in\dC$, $T_z$ is an antilinear normal operator on $\dC^2$. 

Obviously,  $T_z$ is an antilinear self-adjoint operator if and only if $z\in \dR$ and $T_z$ is an antilinear unitary operator if and only if $|z|=1.$
\begin{thl}\label{eigentz}\begin{itemize}
\item[\em (i)] If $z$ is not real, then $T_z$ has no eigenvalue.
\item[\em (ii)]
If $z$ is real, then for any $\varphi\in \dR$ the operator $T_z$ has the eigenvalues $ze^{\i\, \varphi}$ and $T_z$ is the orthogonal sum of two one-dimensional antilinear operators.
\end{itemize}
\end{thl}
\begin{proof}
(i): Assume that $T_z(x_1,x_2)=\lambda (x_1,x_2)$ for some  vector $ (x_1,x_2)\in \dC^2$ and $\lambda\in \dC$. This means that 
$\ov{z}\, \ov{x_2}=\lambda x_1$ and $ z \ov{x_1}=\lambda x_2.$ Then  $\ov{z}\, x_1=\ov{\lambda}\, \ov{x_2}$ and hence $\ov{z}^2 \ov{x_2}\, x_1= |\lambda|^2\, \ov{x_2}x_1.$ Because $z$ is not real, it follows that $\ov{x_2}x_1=0$, so that $x_1=0$ or $x_2=0$. If $x_1=0$, then $\ov{z}\, \ov{x_2}\,  =0$. Since $z\neq 0$, we obtain $x_2=0$. Similarly, $x_2=0$ implies $x_1=0$. Thus, in both cases,  $(x_1,x_2)=0$, that is, $\lambda$ is not an eigenvalue.

(ii): Since $z=\ov{z}$, it follows from (\ref{defTz}) that $$\quad\quad T_z(e^{-i\,\varphi/2},e^{-i\,\varphi/2})=(z e^{i\,\varphi/2},z e^{i\,\varphi/2})=z e^{i\,\varphi}(e^{-i\,\varphi/2},e^{-i\,\varphi/2}).$$ 
Define $S_1( \lambda (1,1))=z\ov{\lambda} (1,1)$ on $\dC\cdot (1,1)$ and $S_2 (\lambda (1,-1))=- z\ov{\lambda}(1,-1)$ on $\dC\cdot (1,-1)$. Then $T_z=S_1\oplus S_2$. 
\end{proof}

Let $\lambda\in \dC$ and $(y_1,y_2)\in \dC^2$. We consider the equation 
\begin{align}
T_z(x_1,x_2)-\lambda (x_1,x_2)=(y_1,y_2),
\end{align}
that is,
\begin{align*}
\ov{z}\, \ov{x_2}-\lambda x_1=y_1, \quad z\ov{x_1} -\lambda x_2=y_2.
\end{align*}
These equations imply that
\begin{align}
(\ov{z}^2 -|\lambda|^2) x_1=\ov{z}\, \ov{ y_2} +\ov{\lambda}\, y_1,\quad (z^2 -|\lambda|^2) x_2=z\, \ov{y_1} +\ov{\lambda} \, y_2.
\end{align}
Therefore, if 
\begin{align}\label{lambdacond}
 |\lambda|\neq z, - z, \ov{z}, -\ov{z},
\end{align}
 then $\lambda$ belongs to the resolvent set of $T_z$ and
\begin{align}\label{resolventz}
x_1=(\ov{z}^2- |\lambda|^2)^{-1}(\ov{z}\, \ov{y_2}+\ov{\lambda}\, y_1),\quad x_2=(z^2-|\lambda|^2)^{-1}(z\, \ov{y_1}+\ov{\lambda}\, y_2).
\end{align}
In particular, we have the following result.

\begin{thl}
If $z$ is not real, then  $\rho(T_z)=  \dC$ and
\begin{align*}\left(\begin{array}{c} x_1\\ x_2
\end{array}\right)=(T_z-\lambda I)^{-1}\left(\begin{array}{c} y_1\\ y_2
\end{array}\right)
\end{align*}
for $\lambda \in \dC$ and $(y_1,y_2)\in \dC^2$ is given by (\ref{resolventz}).
\end{thl}
Formula (\ref{resolventz}) shows that for $\lambda \neq 0$ the operator  $(T_z-\lambda I)^{-1}$ is  neither  linear nor  antilinear, but it is real-linear.

Now we look when two operators $T_z$ and $T_v$ are unitarily equivalent. We compute 
\begin{gather}\label{tztz}
\left(
\begin{array}{ll}
~0 & 1 \\
-1 & 0
\end{array}\right) T_z\left(
\begin{array}{ll}
0 & - 1 \\
1 & ~ 0
\end{array}\right)=T_{-\ov{z}},
\end{gather}\begin{gather}\label{tzt-z} 
\left(
\begin{array}{ll}
1 &~ 0 \\
0 & -1
\end{array}\right) T_z\left(
\begin{array}{ll}
1 & ~0 \\
0 & -1
\end{array}\right)=T_{-z}.
\end{gather}
This shows that $T_z$ is unitarily equivalent to $T_{-\ov{z}}$ and also to $T_{-z}$. Therefore, each operator $T_z$ is unitarily equivalent  to an operator $T_v$ with $v$ from the closed right upper quarter plane
\begin{align*}\ov{K}:=\{v\in \dC: 0\leq \arg (v)\leq \pi/2\},
\end{align*}
where we have set $\arg(0):=0$.
\begin{thl}\label{tzequivalence}
For $z\in \dC$,  $T_z$ is unitarily equivalent to an operator $T_v$ with  $v\in \ov{K}$.

If $z, v\in \ov{K}$  and $z\neq v$, then  $T_z$ and $T_v$ are not unitarily equivalent.
\end{thl}
\begin{proof}
The first assertion has been already noted above, so it remains to prove the second assertion.
Assume to the contrary that $z,v\in \ov{K}$,  $z\neq v$ and $T_z$ and $T_v$ are unitarily equivalent.
Then there is a unitary matrix $U=(u_{jk})_{j,k=1,2}$ with complex entries such that $UT_zU^*=T_v$. 

Then $U(T_z)^\dagger T_zU^*=(T_v)^\dagger T_v$. By (\ref{TdaggerT}) this implies that $|z|=|v|$. Let us abbreviate $\varphi:=\arg(z)$ and $\psi:=\arg(v)$. Further, we have $U(T_z)^2U^*= (T_v)^2$ and hence $U(T_z)^2=(T_v)^2U$.
Using (\ref{Tz2}) this leads to $u_{12}z^2=u_{12} \ov{v}^2$ and $u_{21}\ov{z}^2=u_{12} v^2$.  Now $u_{12}z^2=u_{12} \ov{v}^2$ yields\, $u_{12}e^{\i 2(\varphi+\psi)}=u_{12}$. Since $ 0\leq \varphi\leq \pi/2,$ $0\leq \psi\leq \pi/2$ and $\varphi\neq \psi$, we have $e^{2(\varphi+\psi)}\neq 1$. Therefore $u_{12}=0$. Similarly, $u_{21}=0$, so $U$ is a diagonal matrix. From $UT_z=T_vU$ we obtain 
$u_{11}\, \ov{z}= \ov{v}\, \ov{u_{22}}$ and $u_{22}\, z= v\, \ov{u_{11}}.$ Since 
$|z|=|v|,$ we derive that $u_{11}u_{22}$ is real and equal to $e^{\i (\psi-\varphi)}.$ Since 
$\psi-\varphi\in [-\pi/2,\pi/2]$ (by the asumption $z,v\in\ov{K}$) and $\varphi\neq \psi$, this is a contradiction.
\end{proof}

\section{A Model for Antilinear Normal Operators}\label{nmodel}

Throughout this section, $N_1$ and $N_2$ are normal linear operators on Hilbert spaces $\Hh_1$ and $\Hh_2$, respectively, and $\cN$ denotes the  normal linear operator on $\Hh:=\Hh_1\oplus \Hh_2$ given by the operator matrix
\begin{gather} 
\cN=\left(
\begin{array}{ll}
N_1 & ~ 0 \\
~0 &  N_2
\end{array}\right).
\end{gather}
Further, we assume that $\cC$ be a conjugation on $\Hh_1\oplus \Hh_2$ of the form
\begin{gather} 
\cC=\left(
\begin{array}{ll}
 ~0 & C_2\\
C_1 &  ~0 
\end{array}\right)
\end{gather}
such that 
\begin{align}\label{nccommute}
\cN\cC=\cC\cN.
\end{align}
Then we define an antilinear operator $T$ on $\Hh_1\oplus \Hh_2$ by
\begin{gather}\label{defiT}
T:=\cN\, \cC=\left(
\begin{array}{ll}
~~~ 0 & N_1C_2 \\
N_2C_1 &  ~~~ 0
\end{array}\right).
\end{gather}
This is the setup which  will be kept in this section. In particular,  we assume throughout
that (\ref{nccommute}) holds.

The operator $T$ 
 appears as operator $T_\ns$  in the structure Theorem \ref{normalform}  in Section \ref{structure}. Also, note that the operator $T_z$ defined by (\ref{defTz}) is a special case.

Since $\cC$ is a conjugation, 
 $C_1$ is the adjoint of $C_2:\Hh_2\to \Hh_1$, $C_2$ is the adjoint of $C_1:\Hh_1\to \Hh_2$, and we have
\begin{align}\label{csquared}
C_1C_2=I_{\Hh_2},~~~C_2C_1=I_{\Hh_1}.
\end{align}
Equation (\ref{nccommute}) is equivalent to
\begin{align}\label{nccommute1}
N_1C_2=C_2N_2~~ \textrm{and}~~ N_2C_1=C_1N_1.
\end{align}
Combined with (\ref{csquared}) the latter implies that
\begin{align}\label{n2n1}
C_1N_1C_2=N_2~~ \textrm{and}~~ C_2N_2C_1=N_1.
\end{align}
Clearly, $T$ is closed and densely defined and we have
\begin{gather} 
T^\dagger:=\cC \cN^*=\left(
\begin{array}{ll}
~~~ 0 & C_1N_2^* \\
C_2N_1^* &  ~~~ 0
\end{array}\right).
\end{gather}
It follows  from  $\cN\cC=\cC\cN$ that  $\cC\cN\cC=\cN$  and  hence $\cC\cN^*\cC=\cN^*,$ so that $\cC\cN^*=\cN^*\cC.$ Therefore 
\begin{align}TT^\dagger= \cN \cC\cC\cN^*=\cN\cN^*=\cN^*\cN,~~T^\dagger T=\cC\cN^*\cN\cC=\cN^*\cC \cC\cN=\cN^*\cN.
\end{align}
Hence $T$ is an {\it antilinear normal operator} on $\Hh_1\oplus \Hh_2.$ Further,  (\ref{nccommute})  implies that
\begin{align}\label{T2}
T^2= \cN\cC \cN \cC=\cN \cN \cC\cC=\cN^2.
\end{align}
\begin{thl}
Let $T$ and $T'$ be  operators as above, with corresponding operators $N_1, N_2,  C_1,C_2$ and $N_1', N_2', C_1', C_2'$, respectively. Then $N_1$ and $N_1'$ are unitarily equivalent if and only if $N_2$ and $N_2'$ are.
\end{thl}
\begin{proof}
By symmetry it suffices to prove that if $N_1$ and $N_1'$ are unitarily equivalent, so are $N_2$ and $N_2'$.
Suppose that $V$ is a unitary operator of $\Hh_1$ on $\Hh_1'$ such that $N_1'=VN_1V^*$. Then, $T=C_1'VC_2$ is a unitary operator of $\Hh_2$ on $\Hh_2'$ and using  (\ref{n2n1}) we derive
$$
N_2'=C_1'N_1'C_2=C_1'VN_1V^*C_2=C_1'VC_2N_2C_1V^*C_2=TN_2T^*.
$$
This shows that  $N_2$ and $N_2'$ are also unitarily equivalent.
\end{proof}
\begin{thl}\label{resolventn1n2} $\rho(N_2)=\{ \ov{\lambda}: \lambda \in \rho(N_1)\, \}$ and for $\lambda \in \rho(N_1)$ we have 
\begin{align}\label{rhon1n2}
(N_2-\ov{\lambda})^{-1}=C_1(N_1-\lambda)^{-1}C_2.
\end{align} 
\end{thl}
\begin{proof}
Recall that $N_2=C_1N_1C_2$ by (\ref{n2n1}). Hence, for $x\in \cD(N_2)=C_2\cD(N_1)$, 
\begin{align}\label{n1n2rho}
y:=(N_2-\ov{\lambda} )x=C_1(N_1-\lambda)C_2x.
\end{align}
Therefore, if $\lambda\in \rho(N_1)$, we obtain
 $C_1 (N_1-\lambda)^{-1}C_2y=x$ and from (\ref{n1n2rho}) we conclude 
that $\ov{\lambda}\in \rho(N_2)$ and $C_1 (N_1-\lambda)^{-1}C_2y=(N_2-\ov{\lambda})^{-1}y$ which gives (\ref{rhon1n2}).

 A similar reasoning applies for $\ov{\lambda}\in \rho(N_2)$ and shows that  then $\lambda\in \rho(N_1).$
\end{proof}

For the next considerations restrictions on the spectrum of the operator $N_1$ are useful.  For this  the following definition is convenient.

\begin{thd}\label{msupported}
Let $N$ be a normal linear operator and let $M$ be a  Borel subset of $\dC$. We shall say that $N$ is \emph{$M$-supported} if $E_N(\dC\backslash M)=0$,
where $E_N$ denotes the spectral measure of $N$.
\end{thd}
We will use this notion for the set 
\begin{align}
K:=\{z\in \dC: 0<\arg z\leq \pi/2\} .
\end{align}
(We adopt the convention that the origin $0$ is not in $K$.)
\smallskip

 Now we investigate the resolvent of $T$. Let $\lambda\in \dC$. Given  $(y_1,y_2)\in \Hh_1\oplus \Hh_2$, we consider the equation 
\begin{align}\label{Tlambda}
T(x_1,x_2)-\lambda (x_1,x_2)=(y_1,y_2),
\end{align}
or equivalently,
\begin{align*}
N_1\, C_2x_2-\lambda x_1=y_1, \quad N_2 C_1x_1 -\lambda x_2=y_2.
\end{align*}
Applying $N_1C_2$ to the second equation,  multiplying  the first by $\ov{\lambda}$ and adding  both equations we obtain
 \begin{align}\label{nc1}
(N_1C_2N_2C_1 -\lambda \ov{\lambda}\, )x_1=N_1C_2y_2 + \ov{\lambda}\, y_1.
\end{align}
Similarly, we apply the operator $N_2C_1$ to the first equation, multiply the second  by $\ov{\lambda}$ and add both equations. This yields 
\begin{align}\label{nc2}
(N_2C_1 N_1C_2 -\lambda \ov{\lambda}\, ) x_2=N_2C_1y_1 +\ov{\lambda}\, y_2.
\end{align}
Using  (\ref{nccommute1}) and  (\ref{csquared}) it follows from (\ref{nc1}) and (\ref{nc2}) that
 \begin{align}\label{nc3}
(N_1^2 -|\lambda|^2 )x_1=N_1C_2y_2 + \ov{\lambda}\, y_1,~~~
(N_2^2 -|\lambda|^2 )x_2=N_2C_1y_1 + \ov{\lambda}\, y_2.
\end{align}
By Lemma \ref{resolventn1n2},  $\rho(N_2^2)$ is the complex conjugate of the set  $\rho(N_1^2)$, so   $|\lambda|^2\in \rho(N_1^2)$ implies  $|\lambda|^2\in \rho(N_2^2)$. Therefore, if $|\lambda|^2\in \rho(N_1^2)$
 it follows from (\ref{nc3}) that 
\begin{align}\label{nc4}
x_1&=(N_1^2- |\lambda|^2)^{-1}(N_1C_2 y_2+\ov{\lambda}\, y_1),\\ x_2&=(N_2^2-|\lambda|^2)^{-1}(N_2C_1 y_1+\ov{\lambda}\, y_2),\label{nc5}
\end{align}
and therefore $\lambda \in \rho(T)$. Further, if  $|\lambda|^2\in \sigma(N_1^2)$,  we conclude from (\ref{nc3}) that $\lambda \notin \rho(T)$, so that $\lambda \in \sigma(T)$.

\begin{thp}\label{Tinvertible} 
\begin{itemize}
\item[\em (i)] If   $N_1$ is $K$-supported, then $\cN(T-\lambda I)=\{0\}$  for  $ \lambda \in \dC.$
\item[\em (ii)] If $\dR\subseteq \rho(N_1)$, then $\rho(T)=\dC$ and for $\lambda \in \dC$ the action
 \begin{align*}\left(\begin{array}{c} x_1\\ x_2
\end{array}\right)=(T-\lambda I)^{-1}\left(\begin{array}{c} y_1\\ y_2
\end{array}\right)
\end{align*}
 is given by (\ref{nc4}) and (\ref{nc5}).
\end{itemize}
\end{thp}
\begin{proof}
(i): Let $\lambda \in \dC$. From the assumption that  $N_1$ is $K$-supported it follows in particular that $\cN(N_1\pm|\lambda|\, I)=\{0\}$. This implies  $\cN(N_1^2-|\lambda|^2 \, I)=\{0\}$. Further, since $N_2=C_1N_1C_2$ and hence $N_2^2- |\lambda|^2\, I=C_1 ( N_1^2-|\lambda|^2\, I)C_2$, we conclude that $C_2\cN(N_2^2-|\lambda|^2\, I)\subseteq \cN(N_1^2-|\lambda|^2\, I)=\{0\}$. Thus, $\cN(N_2^2-|\lambda|^2\, I)=\{0\}$.

Now let $(x_1,x_2)\in \cN(T-\lambda I)$. Then $y_1=0, y_2=0$ by (\ref{Tlambda}) and it follows from (\ref{nc3})  that $x_1\in \cN(N_1^2-|\lambda|^2\, I)$ and $x_2\in \cN(N_2^2-|\lambda|^2\, I)$. Therefore, $x_1=0, x_2=0$ by the preceding which proves that $\cN(T-\lambda I)=\{0\}.$ \smallskip

(ii): Clearly,    $\dR\subseteq \rho(N_1)$ implies that $ [0,+\infty)\subseteq \rho(N_1^2)$.  Then, for any $\lambda\in \dC$, we have  $|\lambda|^2\in \rho(N_1^2)$. Therefore, by the  discussion preceding Proposition \ref{Tinvertible},  $\lambda \in \rho(T)$ and the formula for the resolvent follows from  (\ref{nc4}) and (\ref{nc5}).
\end{proof}
 We illustrate this result by a simple example.
\begin{thex}
Suppose  $R$ is a closed subset of the closed right upper quarter plane $\ov{K}$. Let $\Hh_1=L^2(R)$ and $\Hh_2=L^2(R^c)$ with respect to the plane measure of $\dC$. Recall that  $R^c$ denotes the complex conjugate of the set $R$. Define 
\begin{align*}
(N_jf_j)(z)&=z f_j(z)\quad \textrm{for}\quad f_j\in \Hh_j\, , j=1,2,\\ 
\cC(f_1,f_2)&= (\, \ov{f_2}, \ov{f_1}\, )\quad \textrm{for}\quad (f_1,f_2)\in \Hh_1\oplus \Hh_2.
\end{align*} Then the above assumptions are satisfied for the corresponding operator $T$.

The spectral projection $E_{N_1}(M)$ is the multiplication operator by the characteristic function of  $M$. Therefore, since $R\subseteq \ov{K}$,  $N_1$ is $K$-supported and  we have $\cN(T-\lambda I)=\{0\}$  for  $ \lambda \in \dC$ by Proposition \ref{Tinvertible}(i). If  $R$ has a positive  distance  to  the real line  $\dR$, then $\dR\subseteq \rho(N_1)$ and hence $\rho(T)=\dC$  by Proposition \ref{Tinvertible}(ii).
\end{thex}
\begin{thp}\label{uniequi}
Let $T$ and $T'$ be two operators as above, with corresponding operators $N_1, N_2, C_1 ,C_2$ and  $N_1', N_2',C_1',C_2'$  respectively. Suppose  $N_1$ and $N_1'$ are $K$-supported. Then the following are equivalent:
\begin{itemize}
\item[\em (i)] $T$ and $T'$ are unitarily equivalent.
\item[\em (ii)] There are unitary operators $u_{1}$ of $\Hh_1$ on $\Hh_1'$ and $u_{2}$ of $\Hh_2$ on $\Hh_2'$ such that
\begin{align}\label{ucnrealitions}
u_1N_1u_1^*=N_1', ~~ u_2N_2u_2^*=N_2',~~ u_2=C_1'u_1C_2.
\end{align}
\end{itemize}
\end{thp}
\begin{proof}
(i)$\to$(ii):
Let $U$ be a unitary operator  which establishes the unitary equivalence of $T$ and $T'$, that is,  $UTU^*=T'$. Then $U:\Hh_1\oplus \Hh_2\to \Hh_1'\oplus\Hh_2'$ is an  operator block matrix $U=(u_{ij})_{i,j=1,2}$. Clearly, from  $UTU^*=T'$ it follows that  that $UT^2=(T')^2U$ and hence $u_{11}N_1^2=(N_1')^2u_{11}$ by (\ref{T2}). Further, since $N_1$ and $N_1'$ are $K$- supported by assumption, we have\, $E_{N_1}(\dC\backslash K)=$ and\, $E_{N_1'}(\dC\backslash K)=0$. 

First we show that $u_{11}N_1=N_1' \,u_{11}$. Since $E_{N_1}(\dC\backslash K)=$, we have  
$$N_1=\int_{K} zdE_{N_1}(z)\quad \textrm{and}\quad N_1^2=\int_{K} z^2dE_{N_1}(z).$$
Since the map $z\to z^2$ is a bijection of $K$ and $K^2$, it follows from the transformation formula for spectral integrals (\cite[Proposition 4.24]{schm12}) that for each Borel set $M$ of $K$ the spectral projection $E_{N_1^2}(M^2)$ of the normal operators $N_1^2$ is equal to 
$E_{N_1}(M)$. Hence the sets of spectral projections of $N_1^2$ and $N_1$ coincide. Similarly, the sets of spectral projections of $(N_1')^2$ and $N_1'$ are the same. Hence it follows from $u_{11}N_1^2=(N_1')^2u_{11}$ that $u_{11}E_{N_1}(M)=E_{N_1'}(M)u_{11}$ for all Borel subsets $M$ of $K$  and therefore $u_{11}N_1=N_1' \,u_{11}$ by  \cite[Proposition 5.15]{schm12}.

Next we prove that $u_{12}=0$.  First we note that from $UT=T'U$ we obtain  $u_{12}N_2=N_1'u_{12}$.
Applying once again \cite[Proposition 5.15]{schm12}, it follows that
\begin{align}\label{Einter} 
u_{12}E_{N_2}(M)=E_{N_1'}(M)u_{12}
\end{align}
for all Borel subsets $M$ of $\dC$. Set $M:=\{z\in \dC:3\pi/2\leq\arg(z)\leq 2\pi\}$. Since $M\cap K=\emptyset$ and $N_1'$ is $K$-supported, $E_{N_1'}(M)=0$. Because $N_1$ is $K$-supported, the spectrum of $N_1$ is contained in $\ov{K}$. Therefore, by Lemma \ref{rhon1n2}, the spectrum of $N_2$ is contained in the complex conjugate  of $\ov{K}^c$ of the set $\ov{K}$ which is the set $M$. Hence $E_{N_2}(M)=I$.
Inserting these facts  into (\ref{Einter}) we obtain $u_{12}=0$.

A similar reasoning proves that $u_{21}=0$. Indeed, $u_{21}N_1=N_2'u_{21}$ yields $u_{12}E_{N_1}(M)=E_{N_2'}(M)u_{12}$. Setting $M$ as above, we get $u_{21}=0$.

Thus, the unitary operator $U$ is diagonal. Hence  $u_1:=u_{11}$ is a unitary operator of $\Hh_1$ on $\Hh_1'$ and $u_2:=u_{22}$ is a unitary operator of $\Hh_2$ on $\Hh_2'$. Further,  from $UTU^*=T'$ we obtain that $u_1N_1u_1^*=N_1'$ and $ u_2N_2u_2^*=N_2'.$ 

The two latter facts  imply that $U\cN U^*=\cN'$. Combined with $UTU^*=T'$ it follows that $U\cC U^*=\cC'$. Hence $U=\cC' U\cC$ which gives $u_2=C_1'u_1C_2$.\smallskip

(ii)$\to$(i):  Let $U$ be the diagonal block matrix with diagonal entries $u_1$ and $u_2$. The first two equalities of (\ref{ucnrealitions}) yield $U\cN U^*=\cN'$. Clearly, $u_2=C_1'u_1C_2$ implies $u_1=C_1'u_2C_1$, so that $U=\cC' U\cC$ and hence $U\cC U^*=\cC'$. Combined with 
$U\cN U^*=\cN'$ the latter gives $UTU^*=T'$.
\end{proof}

\begin{thd} An operator $S$ on $\Hh$ is  said to be \emph{irreducible} if there is no orthogonal decomposition $S=S_1\oplus S_2$ on $\Hh= \cG_1\oplus \cG_2$, where $S_1, S_2$ are operators on $\cG_1, \cG_2$, respectively, and $\cG_1, \cG_2\neq \{0\}.$
\end{thd}
\begin{thp}\label{irre}
\begin{itemize}
\item[\em (i)] If $T$ is irreducible, then $\dim \Hh_1=1$, or equivalently, $\dim \Hh_2=1$.   
\item[\em (ii)] If $\dim \Hh_1=1$ and $N_1$ on $\Hh_1$ is not real, then $T$ is irreducible.
\item[\em (iii)] If $T$ is irreducible, then $T$ is unitarily equivalent to an operator $T_z$ defined by (\ref{defTz}), where $z\in \dC$, $0<\arg z\leq \frac{\pi}{2}$.
\end{itemize}
\end{thp}
\begin{proof}
(i): First we note that  $\dim \Hh_1=\dim \Hh_2$, because $C_2$ is an  antilinear bijection of $\Hh_2$ on $\Hh_1$.

Suppose   $\dim\Hh_1>1$. From the spectral theorem it follows that the normal operator $N_1$ on $\Hh_1$ decomposes as an orthogonal direct sum $N_1=N_{11}\oplus N_{12}$ on $  \Hh_1=\Hh_{11}\oplus \Hh_{12}$, where $\dim \Hh_{11} \geq1$ and $\dim \Hh_{12}\geq 1$. Set $\cG_1:=\Hh_{11}\oplus C_1 \Hh_{11}$. Then $\cC(\cG_1)=C_2C_1\Hh_{11}\oplus C_1 \Hh_{11}=\cG_1$, so the closed linear subspace $\cG_1$ of $\Hh$ is invariant under $\cC$. Recall that  $N_2=C_1N_1C_2$ by (\ref{n2n1}). Hence we have $\cD(N_2)=C_1\cD(N_1)$ and $N_2$ decomposes as $N_2=N_{21}\oplus N_{22}$ on $\Hh_2=C_1 \Hh_{11}\oplus (C_1\Hh_{11})^\bot$. The preceding implies that $T$ decomposes as an orthogonal sum $T=T_1\oplus T_2$ on 
$\Hh= \cG_1\oplus (\cG_1)^\bot$,  with $\cG_1, \cG_2\neq \{0\}.$ This shows that $T$ is not irreducible.

(ii): Let $N_1=z_1$ on $\Hh_1\cong \dC$ and $N_2=z_2$ on $\Hh_2\cong \dC$. Assume that  $\dC\cdot (x_1,x_2)$  is a one-dimensional linear subspace of $\Hh$ which is invariant under $T$. Then
$T(x_1,x_2)=(z_1C_2x_2,z_2 C_1x_1)$ is a multiple of $(x_1,x_2)$, so there exists $\alpha\in \dC$ such that 
\begin{align}\label{zcrela}z_1C_2x_2=\alpha x_1~~ \textrm{and}~~  z_2C_1 x_1=\alpha x_2.
\end{align}
The relation $N_2=C_1N_1C_2$ yields $z_2=\ov{z_1}$. Using this fact
, the equality $C_1C_2=I_{\Hh_2}$ 
and (\ref{zcrela}) we derive
$$
\ov{z_1}^2\, x_2=z_2 \ov{z_1}\, x_2=z_2 C_1(z_1 C_2x_2)=z_2 C_1(\alpha x_1)=  \ov{\alpha} \, (z_2 C_1x_1)= \ov{\alpha}\, \alpha\, x_2.
$$
Since $z_1$ is not real,  (\ref{zcrela}) implies  $x_2\neq 0$. Hence $\ov{z_1}^2= |\alpha|^2$,  a contradiction, because $z_1$ is not real. Hence $T$ is irreducible.
\smallskip

(iii): Since $T$ is irreducible, by (i) and (ii) we can assume that $\Hh_1=\Hh_2=\dC$ and $N_1=z_1$, $N_2=\ov{z_1}$. 
Let $C_1'$ and $C_2'$ denote the complex conjugation on $\Hh_1$ and $\Hh_2$ and set $N_1'=\ov{z}$, $N_2=z$, where $z:=\ov{z_1}$. Then the operator $T'$ with $\cN', \cC'$ is just the operator $T_z$ given  by (\ref{defTz}). Since $\cC$ is a conjugation,  $|C_2(1)|=1$. Setting $u_1:=1, u_2:=\ov{C_2(1)}$, we have  $u_2=C_1' u_1C_2$ and the diagonal operator with diagonal entries $u_1$, $u_2$ provides the unitary equivalence of $T$ and $T'=T_z$, see
(\ref{ucnrealitions}). Lemma \ref{eigentz}(ii) implies that  $z$ is not real. Since $T_z$ is unitariy equivalent to $T_{-\ov{z}}$ and $T_{-z}$ by (\ref{tztz}) and (\ref{tzt-z}), respectively, we can assume that $0<\arg z\leq \frac{\pi}{2}.$
\end{proof}

Let $S$ be a linear or antilinear operator on  $\Hh$. If $S_0$ is a self-adjoint linear resp. antilinear operator acting a closed subspace $\Hh_0$ of $\Hh$ such that $S_0\subseteq S$, we say that $S_0$ is a {\it self-adjoint part} of $A$. Note that in this case  $\cD(S_0)\subseteq \cD(S)$, $\cD(S_0)$ is dense in $\Hh_0$ and $S_0$ maps $\cD(S_0)$ into $\Hh_0$. The self-adjoint operator $S_0=0$  on $\cD(S_0)=\Hh_0=\{0\}$ is called the trivial  self-adjoint part of $S$.
\begin{thd}\label{defnsa}
A linear or antilinear operator $S$ on a Hilbert space $\Hh$ is called \emph{completely non-self-adjoint} if  the only self-adjoint part of $S$ is the trivial part.
\end{thd}
The following  fact will be used later several times.
\begin{thl}\label{auxilemma}
 $N_2=C_1N_1C_2$ implies that $E_{N_2}(M)=C_1E_{N_1}(M^c)C_2$ for each Borel subset $M$ of $\dC$.
\end{thl}
\begin{proof}
We first note that $\widetilde{E}$ defined by $\widetilde{E}(M):=E_{N_1}(M^c )$ is the spectral measure of $N_1^*$ and  $N_1=\int \ov{z}\, d \widetilde{E} (z) $. Then for $x,y\in \cD(N_2)=C_1 \cD(N_1)$ we derive
\begin{align*}
& \langle  N_2 x,y\rangle =\langle C_1N_1C_2 x,y\rangle=  \langle  C_1y, N_1C_2x\rangle = \\ \langle C_1 y ,\int \ov{z}\, d\langle & \widetilde{E}(z) C_2 x \rangle=\int z\, d\langle C_1 y, \widetilde{E} (z) C_2 x \rangle =\int z\, d \langle C_1 \widetilde{E} (z)C_2 x,y\rangle .
\end{align*}
Therefore, by the uniqueness of the spectral measure, the spectral measure $C_1\widetilde{E} C_2$ is equal to the spectral measure $E_{N_2}$ of the normal operator $N_2$. Thus, $E_{N_2}(M)=C_1E_{N_1}(M^c)C_2.$
\end{proof}

\begin{thp}\label{nsap}
Let $T$ be the antilinear normal operator defined as above. If $N_1$ is $K$-supported, then $T$ completely non-self-adjoint.
\end{thp}
\begin{proof}
Let $T_0$ be a self-adjoint part of $T$ acting on a Hilbert subspace $\Hh_0$. By Proposition \ref{adag}, $T_0^2$ is a positive self-adjoint linear operator on $\Hh_0$. Further,  $T^2$ is a normal linear operator on $\Hh$. Since $T_0\subseteq T$, we have  
\begin{gather} \label{t2}
T_0^2\subseteq T^2=\left(
\begin{array}{ll}
N_1^2 & ~ 0 \\
~ 0 &  N_2^2
\end{array}\right).
\end{gather}
 Because $T_0^2$ is self-adjoint on $\Hh_0$, there exists a linear operator $S$ on $\Hh_0^\bot$ such that $T^2 =T_0^2\oplus S$ on
 $\Hh=\Hh_0\oplus \Hh_0^\bot$ (\cite[Propositiom 1.17]{schm12}). Since $T^2$ is normal on $\Hh$, $S$ is normal  on $\Hh_0^\bot$ as well. For the  spectral projections we have\begin{align}\label{directsumee}
E_{T^2}(\cdot)=E_{T_0^2}(\cdot)\oplus E_{S}(\cdot)\quad \textrm{on}\quad \Hh=\Hh_0\oplus \Hh_0^\bot.
\end{align}
On the other hand, it follows from (\ref{t2}) that the spectral projections of $T^2$ decompose as
\begin{align}\label{directsumee1}
E_{T^2}(\cdot)=E_{N_1^2}(\cdot)\oplus E_{N_2^2}(\cdot)\quad \textrm{on}\quad \Hh=\Hh_1\oplus \Hh_2.
\end{align}
By assumption,  $N_1$ is $K$-supported.  Arguing as in the proof of
Proposition \ref{uniequi}, we conclude that $N_1^2$ is supported by the set $\{z\in \dC:0<{\rm arg}\, z\leq \pi\}$. Since $N_2=C_1N_1C_2$, it follows from Lemma \ref{auxilemma} that $E_{N_2}(M)=C_1E_{N_1}(M^c )C_2$ for each Borel subset $M$ of $\dC$. Hence $N_2$ is supported by $\{z\in \dC: \pi\leq {\rm arg}\, z< 2\pi\}$. Thus we obtain  from (\ref{directsumee1}) that $E_{T^2}([0,+\infty))=0$ and hence $E_{T_0^2}([0,+\infty))=0$ by (\ref{directsumee}). Since $T_0^2$ is a positive self-adjoint operator, this is only possible if $\Hh_0=\{0\}$, that is, $T_0$ is the trivial self-adjoint part. 
\end{proof}
\begin{rem}
A linear normal operator $A$ on $\Hh$ is completely non-self-adjoint according to Definition \ref{defnsa} if and only if 
$E_A(\dR)=0$.
\end{rem}

\section{A Structure Theorem for Antilinear Normal Operators}\label{structure}

Let us recall two definitions which have been used already earlier in this paper. Throughout we abbreviate
\begin{align}
K:=\{z\in \dC: 0<\arg z\leq \pi/2\}~~ \textrm{and}~~ K_\dT:=K\cap \dT.
\end{align}
According to Definition \ref{msupported}  a normal operator  $N$ is called \emph{$K$-supported} if $$E_N(\dC\backslash K)=0.$$

The following theorem is the main result of this paper. By an {\it antinormal operator} we mean  an antilinear normal operator.

\begin{tht}\label{normalform}
Suppose that $T$ is a densely defined closed antilinear  operator on a Hilbert space $\Hh$. Then the following are equivalent:
\begin{itemize}
\item[\em (i)] \, $T$ is antinormal. 
\item[\em (ii)] \, $T$ is of the form 
\begin{align}\label{Tsumsansa}
T=T_\sa\oplus T_\ns ~~~ \textrm{on}~~~\Hh= \Hh_\sa\oplus \Hh_\ns,~~\textrm{where}
\end{align} 
$\bullet$\, $T_\sa =SC_0=C_0S$ for some positive self-adjoint linear operator $S$ and some conjugation $C_0$ on  $\Hh_\sa$  and\\ 
$\bullet$\, $T_\ns=\cN \cC=\cC\cN$ 
for some normal operator  $ \cN=\left(
\begin{array}{ll}
 N_1 & ~0\\
~0 &  N_2
\end{array}\right)$, with normal linear operators $N_1$ and $N_2$ on $\Hh_1$ and $\Hh_2$, respectively, such that $N_1$ is $K$-supported, and  some conjugation $ \cC=\left(
\begin{array}{ll}
 ~0 & C_2\\
C_1 &  ~0 
\end{array}\right)$ on 
$\Hh_\ns= \Hh_1\oplus \Hh_2.$ 
\end{itemize}
\end{tht}

Theorem \ref{normalform} will be proved in Section \ref{proof}. Here we  
discuss this result,  derive some consequences, and collect a number of useful facts and formulas for an antilinear normal operator $T$. For this we retain the notation of Theorem \ref{normalform}.

First we note that the operator $T_\ns$ is just the operator model treated in Section \ref{nmodel}.  Therefore, since $N_1$ is supported by the set $K$, all results and formulas from this section apply to $T_\ns$. Let  us briefly recall some of them.

 First, $T_\ns$ is completely non-self-adjoint according to Proposition \ref{nsap} which explaines our notation $T_\ns$.  Formula  (\ref{n2n1}) yields
\begin{align*}
 N_2=C_1N_1C_2\quad \textrm{and}\quad  N_1= C_2N_2C_1.
\end{align*}
Next,  by Proposition \ref{Tinvertible} we have
\ $\cN(T_\ns-\lambda I)=\{0\}$  for all $ \lambda \in \dC.$  Even more, under the stronger condition $\sigma (N_1)\subseteq K$ the resolvent set $\rho(T_\ns)$ 
is whole complex plane $\dC$ and the resolvent of $T_\ns$ is given by the formulas (\ref{nc4}) and (\ref{nc5}).
Finally,
Proposition \ref{uniequi} is about the unitary equivalence of operators $T_\ns$ and Proposition \ref{irre}  characterizes when  $T_\ns$ is irreducible.

Next we state some useful formulas.
On the Hilbert space  $\Hh= \Hh_\sa\oplus \Hh_\ns\equiv \Hh_\ns\oplus(\Hh_1\oplus \Hh_2)$ we define a conjugation 
$$C =C_0\oplus \cC= C_0\oplus \left(
\begin{array}{ll}
 ~ 0 & C_2\\
C_1 &  ~ 0
\end{array}\right)$$
and a normal linear operator 
$$ N=S\oplus \cN= S\oplus \left(
\begin{array}{ll}
  N_1 & ~ 0\\
 ~ 0 & N_2
\end{array}\right).$$
 From $T_\sa =SC_0=C_0S$ and $T_\ns =\cN\cC=\cC\cN$ we obtain
\begin{align}\label{formNC}T =NC=CN=T_\sa\oplus T_\ns= SC_0\oplus \left(
\begin{array}{ll}
  ~ 0 & N_1C_2\\
N_2C_1 & ~0 
\end{array}\right).
\end{align}
Formula (\ref{formNC}) describes the structure of the antilinear normal operator $T$.

Since $N_1$ is supported by $K$, the spectral measure $E_N$ of the normal operator $N$ is supported by the closed right half plane $H:=\{z\in \dC: {\rm Re\, z}\, \geq 0\}$. Then
$$\Hh=\Hh_\sa\oplus \Hh_1\oplus \Hh_2=E_N([0,+\infty))\Hh\oplus E_N(K)\Hh\oplus E_N(K^c)\Hh$$ and $N$ is given by
\begin{align}
N=S \oplus N_1\oplus N_2\equiv \int_0^{+\infty}t\, dE_N(t) \oplus \int_K z\, dE_N(z) \oplus \int_{K^c} z\, dE_N(z).
\end{align}
The conjugation $C$ leaves $\Hh_\sa$ invariant and gives a bijection of $\Hh_1$ and $\Hh_2$.
Moreover, 
\begin{gather}\label{matrixtbt2}
(T_\ns)^\dagger =\left(
\begin{array}{ll}
~~~ 0 & N_1^*C_2\\
N_2^*C_1 & ~~~0
\end{array}\right), \quad (T_\ns)^2=\left(
\begin{array}{ll}
N_1^2 & ~~~0\\
0 & N_2^2
\end{array}\right), \\
  (T_\ns)^\dagger T_\ns =\left(
\begin{array}{ll}
N_1^*N_1 & ~~~0\\
~~~ 0 & N_2^*N_2
\end{array}\right),\quad | T_\ns |=\left(
\begin{array}{ll}
 |N_1| & ~~0\\
 ~~0 & |N_2|
\end{array}\right).\label{matrixtbt3}
\end{gather}
Since $T_\ns=\cN\cC=\cC\cN$ implies $T_\ns ^\dagger=\cN^*\cC=\cC\cN^*$, we compute for $k,n\in \dN_0$,
\begin{gather*}
((T_\ns)^\dagger)^k(T_\ns)^n =(\cN^*)^k\cN^n\cC^{k+n}=\left(
\begin{array}{ll}
(N_1^*)^kN_1^n & ~~~0\\
0 & (N_2)^*)^kN_2^n
\end{array}\right) \cC^{k+n} .
\end{gather*}
Note that $\cC^{k+n}=I_{\Hh_\ns }$ if $k+n$ is even and  $\cC^{k+n}=\cC$ if $k+n$ is odd.

Since $\cN=\cC\cN\cC$, it follows from Lemma \ref{auxilemma} that for the spectral measure $E_\cN$ of  $\cN$ and each Borel subset $M$ of $\dC$ we have
 \begin{align}\label{enc}\cC E_\cN (M)\cC=E_\cN(M^c).
\end{align}

Now we formulate some consequences of Theorem \ref{normalform}.
First we restate the special case of Theorem \ref{normalform} for antiunitaries.

\begin{thp}
An antilinear  operator $T$  on $\Hh$ is unitary if and only if $T$ is the form (\ref{formNC}), where $S=I_{\Hh_\sa}$ and $N_1$ and $N_2$ are  unitary linear operators on $\Hh_1$ and $\Hh_2$, respectively, such that $E_{N_1}$ is supported by $K_\dT$.
\end{thp}
\begin{proof}
Obviously, $T$ is an antiunitary if and only if $T_\sa$ and $T_\ns$ are.  Clearly, this holds if and only if $S=I_{\Hh_\sa}$ (by $S\geq 0$) and $N_1$ and $N_2$ are  unitaries.
\end{proof}
The next proposition deals with irreducible antilinear normal operators.
\begin{thp}
Each irreducible antilinear normal
 operator $T$ is unitarily equivalent to an operator of the following list. Two operators of this list are not unitarily equivalent.
\begin{itemize}
\item[\em (i)] $S_t=t$ defined by (\ref{defist}) on $\dC$ for some $t\in [0,+\infty)$.
\item[\em (ii)] $T_z$ defined by (\ref{defTz}) on $\dC^2$ for $z\in \dC$, $0<\arg z\leq \frac{\pi}{2}.$
\end{itemize}
\end{thp}
\begin{proof} Since $T=T_\sa\oplus T_\ns$ by (\ref{Tsumsansa}),  the assertion follows at once from Propositions \ref{irrselfadjoint} and
 \ref{irre}.
\end{proof}
Finally, we characterize the unitary equivalence of antinormal operators.
\begin{thp}
Let $T$ and $T'$ be two antinormal operators with corresponding operators $N_1, N_2, C_1, C_2, C_0$ and $N_1', N_2', C_1', C_2',C_0'$, respectively, as in Theorem \ref{normalform}.
Then $T$ and  $T'$ are unitarily equivalent if and only if there exists unitary operators $u_0,u_1, u_2$ of $\Hh_\sa, \Hh_1, \Hh_2$ on $\Hh_\sa', \Hh_1', \Hh_2'$, respectively, such that
\begin{align*}
u_0 Su_0^{-1} =S',\, u_1 N_1u_0^{-1} =N_1', \, u_2 N_2u_2^{-1} =N_2', u_0 C_0 u_0^{-1} =C_0', u_2 =C_1' u_1 C_2. 
\end{align*}
\end{thp}
\begin{proof}
Clearly,  $T$ and $T'$ are unitary equivalent if and only their self-adjoint parts and their non-self-adjoint parts are. For the self-adjoint parts this holds if and only if $u_0 Su_0^{-1}=S'$, $ u_0 C_0 u_0^{-1} =C_0'$. For the non-self-adjoint  parts we can apply  Proposition \ref{uniequi}, because $N_1$ and $N_1'$ are $K$-supported by 
Theorem \ref{normalform}.
\end{proof}
\section{Proof of Theorem \ref{normalform}}\label{proof}
The following auxiliary lemma is used in the proof of Theorem \ref{normalform}.
 It is a slight generalization of a lemma in \cite{galindo}:

\begin{thl}\label{galindo}
Let $C$ be a conjugation or  an skew-conjugation on a Hilbert space $\cG$ and $S$ a bounded self-adjoint operator on $\cG$. Suppose   that  $CS=S C$ and 
\begin{align}\label{asscbotx}
\langle  CS^m x,x\rangle=0 ~~~\textrm{for all}~~x\in \cG, ~ m\in \dN_0.
\end{align} Then there exists a closed linear subspace $\cR$ of $\cG$ such that $S\cR\subseteq \cR$ and $$\cG=\cR\oplus C\cR.$$
\end{thl}
\begin{proof}
Let $\cS$ be the collection of closed subspaces $\cR$ of $\cG$ such that $\cR\bot C\cR$ and $S\cR\subseteq \cR$ . A straightforward application of Zorn's lemma shows that $\cS$ has a maximal element,  say $\cR_0$, with respect to the inclusion. Let  $\cG_0:=\cR_0\oplus C\cR_0.$ It suffices to show that $\cG_0=\cG$.

Suppose that $x\in \cG_0^\bot.$ Let $\cR_x$ denote the smallest $S$-invariant linear subspace which contains $x$ and let $\cG_1$ be the closure of $\cR_x+ \cR_0$. Clearly, $\cR_x$ and $\cG_1$ are invariant under $S$. We show that $\cR_1\bot C\cR_1$. For this it suffices to prove that 
\begin{align}\label{gbotg}
\cR_x+\cR_0\, \bot \, (C\cR_x+C \cR_0).
\end{align}

Clearly, $\cR_0\bot C\cR_0$, because $\cR_0\in \cS$. Since $SC=CS$ and $\cR_0$ is $S$-invariant, $C\cR_0$ is $S$-invariant as well. Thus, $\cG_1$ is $S$-invariant. Since $S$ is self-adjoint and bounded, $\cG_1^\bot$ is also $S$-invariant.  Hence, since $x\in \cG_0^\bot$, $\cR_x\subseteq \cR_0^\bot$, so that  $\cR_x\bot \cR_0$ and $\cR_x\bot C\cR_0$. Therefore, since $C$ is a conjugation or an anto-conjugation, $C\cR_x\bot \cR_0$.

 It remains to prove that $C\cR_x\bot \cR_x$. Using the assumptions on $S$, especially (\ref{asscbotx}), we derive
$$
\langle CS^kx,S^nx\rangle =\langle S^n CS^k x,x\rangle =\langle CS^{k+n}x,x\rangle =0,
$$
so that $CS^k x\bot S^n x$ for $k,n\in \dN_0$. 
This implies that $C\cR_x\bot \cR_x$.

In the preceding two paragraphs we have proved that (\ref{gbotg}) holds. Therefore, $\cR_1\in \cS$ and $\cR_0\subseteq \cR_1.$ From the maximality of $\cR_0$ it follows that  $\cR_0=\cR_1$. Hence $\cG_0^\bot=\{0\}$, that is, $ \cR_0\oplus C\cR_0\equiv \cG_0=\cG$.
\end{proof}

The implication (ii)$\to$(i) of Theorem \ref{normalform} is easily verified.
Now we begin with the proof of the main implication (i)$\to$(ii).

By Theorem \ref{charnpolar}, $T^2$ is a normal linear operator and there exist an antiunitary operator $V_T$ and a positive self-adjoint linear operator $|T|$ on $\Hh$ such that
\begin{align}\label{tsquared}
T^2=V_T^2 |T|^2=|T|^2V_T^2 .
\end{align}
Since $V_T$ is  an antilinear unitary operator, $V_T^2$ is a unitary linear operator on $\Hh$. 

Our first main aim is to  desribe the structure of the antiunitary operator $V_T$, that is, we treat the case  $T=V_T$. The crucial ingredient is the spectral theory of the linear unitary  operator $V_T^2$.
By the spectral theorem  there is a spectral measure 
$E_{V_T^2}$ on $\dT$ such that 
\begin{align}
V_T^2=\int_{\dT}  z\, dE_{V_T^2}(z).\label{spectralvt2}
\end{align}

 Let $\lambda \in \rho(V_T^2)$. Using that $V_T$ is antilinear  we derive
\begin{align*}
V_T(V_T^2-\lambda I)x=(V_T^2-\ov{\lambda} I)V_T x
\end{align*}
 for all $x\in \Hh$. Setting $y=(V_T^2-\lambda I)x$ we obtain
\begin{align}\label{tresolvent}
(V_T^2-\ov{\lambda} I)^{-1} V_Ty=V_T(V_T^2-\lambda I)^{-1} y \quad \textrm{for all}~~y\in \Hh, \,\lambda \in \rho(V_T^2).
\end{align}
Recall that $M^c:= \{ \ov{z}: z\in M\}$ denotes the complex onjugate of a set $M$.
From (\ref{tresolvent}) it follows that for each Borel subset $M$ of $\dT$ we have
\begin{align}\label{EVA2}E_{V_T^2}(M^c )\, V_T=V_T E_{V_T^2}(M).
\end{align}

First we treat the restriction of $V_T$ to the subspace $$\cG:=\cN(V_T^2+I)= \E( \{-1 \})\Hh.$$ The map $J:= V_T\lceil \cG$ of the Hilbert space $\cG$ into itself is antilinear and isometric. Since $J^2x=V_T^2x=-x$ for $x\in \cG$,  we have $J^2=-I_\cG$, that is, $J$ is a skew-conjugation on $\cG$. 

For later use  we suppose that $S$ is a bounded self-adjoint linear operator on $\cG$ which commutes with $J$. When we treat the case of a general antilinear normal operator in the later stages of the proof  this operator $S$ will be specified. 
Using that $V_T$ is antiunitary we derive
\begin{align*}
\langle & JS^m x,x\rangle=\langle Jx,J^2S^m x\rangle=\langle Jx,S^mJ^2x\rangle\\ &=\langle Jx,-S^m x
\rangle =- \langle S^mJx,x\rangle =- \langle JS^mx,x\rangle,
\end{align*}
so that $\langle JS^mx,x\rangle=0$ for $x\in \cG$ and $m\in \dN_0$.
This shows that the assumptions of Lemma \ref{galindo} are satisfied. By this result there exists a closed linear subspace $\cR$
of $\cG$ such that $S\cR\subseteq \cR$,  $\cR\, \bot\, J\cR$ and $\cG=\cR\oplus J \cR$. We define closed linear subspaces  $\cK_1$ and $\cK_2$  of $\cG$ by
\begin{align}\label{defk1k2}
\cK_1:=J\cR\equiv  V_T\cR, \quad \cK_2:=\cR.
\end{align}
Then  $\cG\equiv \cN(V_T^2+I)=\cK_1\oplus \cK_2$. Clearly,
$V_T:\cK_1\to \cK_2$ and $V_T:\cK_2\to \cK_1$.

Now we define 
\begin{align}\label{defhsahnsa}
\Hh_\sa :=E_{V_T^2}(\{ 1\})\Hh~~ \textrm{and}~~\Hh_\ns:=E_{V_T^2}(\dT \backslash \{1\})\Hh.
\end{align}
Since we consider the case $T=V_T$, we then have $\Hh=\Hh_\sa \oplus \Hh_\ns.$

Let us abbreviate
$$C_+:=\{z\in \dT:0<\arg z<\pi\},~~ C_-:=\{z\in \dT:\pi<\arg z<2\pi\}.$$ 
We  define two  unitary linear operators  $U_1$ and $U_2$  on the Hilbert subspaces 
\begin{align} \Hh_1 :=\E(C_+)& \Hh\oplus  \cK_1,~ \Hh_2:=\E(C_-)\Hh\oplus \cK_2,
\end{align}
respectively,  of $\Hh_\ns$ by
\begin{align} U_1:=&\int_{C_+}  z_+^{1/2} \, d\E(z)\oplus {\rm i}\,I_{\cK_1},\label{u1} \\ U_2:=&\int_{C_-}  z_+^{1/2} \, d\E(z)\oplus (- {\rm i})I_{\cK_2},\label{u2} ,
\end{align}
where $z_+^{1/2}$ denotes the square root of $z\in \dC\backslash \dR$ which is uniquely determined by the property that $\sign~ \Imt\, z =\sign~ \Imt\, z_+^{1/2}$. 
Clearly, $\Hh_1\bot \Hh_2$  and since $\cK_1\oplus \cK_2=\cN(V_T^2+I)=\E(\{-1\})\Hh$, 
we have  $$\Hh_\ns =\Hh_1\oplus \Hh_2.$$

Let $U_1'$ and $U_2'$ denote the operator integrals in (\ref{u1}) and (\ref{u2}) acting on the corresponding Hilbert subspaces $\Hh_1'$ and $\Hh_2'$, respectively.

 From the transformation formula of spectral measures (see e.g. \cite[Proposition 4.24]{schm12}) it follows that the spectral measure of the unitary operator $U_1'$ on $\Hh_1'$ is supported by the set $\{z\in \dT:0<\arg z<\pi/2\}$. Since $U_1\lceil \cK_1=
{\rm i}\, I_{\cK_1}$, $U_1$ is supported  by  $\{z\in \dT:0<\arg z\leq \pi/2\}=K_\dT$.

Let $x\in \Hh_1'$. We consider  the integral $U_1'x$ as a limit of sums of terms  $(z_j)_+^{1/2} E_{V_T^2}(M_j)x$, with $z_j\in C_+$ and $M_j\subseteq C_+$.  Using (\ref{EVA2}) and the fact that $V_T$ is antilinear 
 we obtain   
\begin{align*}
V_T  \big((z_j)_+^{1/2} E_{V_T^2}(M_j)\big) x=\big( \ov{z_j}_+^{1/2} E_{V_T^2}(M_j^c )\big) V_T x.
\end{align*}
Since $\ov{z_j}_+^{1/2}  E_{V_T^2}(M_j^c )V_T x$ are  the corresponding terms of a sum for the integral $U_2'(V_T x)$, we conclude that $V_Tx\in \Hh_2'$ and $V_TU_1'x=U_2' V_T x$ for $x\in \Hh_1'$.
Further, $V_T$
maps  ${\rm i}x \in \cK_1$ to $(-{\rm i}) V_1x\in \cK_2$ for $x\in \cK_1$. Therefore $V_T:\Hh_1\to \Hh_2$ and
\begin{align}\label{uvrelation1}
V_TU_1x=U_2V_Tx, \quad x\in \Hh_1.
\end{align}
A similar reasoning yields that $V_T:\Hh_2\to \Hh_1$ and
\begin{align}\label{uvrelation2}
V_TU_2y=U_1V_Ty, \quad y\in \Hh_2.
\end{align}

Let $V_\ns$ denote the restriction of $V_T$ to $\Hh_\ns$. Then $V_\ns:\Hh_\ns\to \Hh_\ns$ and 
  the antlinear operator $V_\ns$ on $\Hh_\ns$
 is given by a operator block matrix
\begin{gather} 
V_\ns=\left(
\begin{array}{ll}
~0 & V_2\\
V_1 & ~0
\end{array}\right).
\end{gather}
From their definitions it is clear that $U_1$ and $U_2$ are linear unitary operators on the Hilbert spaces $\Hh_1$ and $\Hh_2$, respectively.
Let $U$ denote the unitary linear operator on $\Hh_\ns=\Hh_1\oplus \Hh_2$ defined by the  block matrix
\begin{gather}\label{defiU}
U=\left(
\begin{array}{ll}
U_1 & ~ 0\\
~ 0 & U_2
\end{array}\right).
\end{gather}
From (\ref{uvrelation1}) and (\ref{uvrelation2}) we obtain
\begin{align}\label{uvrelation3}
V_\ns U=UV_\ns.
\end{align}

Next we  verify that
\begin{align}\label{vnsu2}
V_\ns^2 =U^2.
\end{align}
Indeed, since $V_T^2x=-x$ for $x\in \cK_1$, it follows from the definitions of $U_1$ and $\Hh_1$ and the spectral decomposition (\ref{spectralvt2}) of $V_T^2$ that 
\begin{align}
U_1^2=\int_{C_+}  z \, d\E(z)\oplus ( -I_{\cK_1}) =V_T^2\lceil \Hh_1=V_\ns^2\lceil \Hh_1.
\end{align}
Similarly, $U_2^2=V_\ns^2\lceil \Hh_2$. Since $\Hh_\ns=\Hh_1\oplus \Hh_2$,  these equalities imply (\ref{vnsu2}).

Define
\begin{gather} 
\cC:= V_\ns U^{-1}=\left(
\begin{array}{ll}
~~ 0 & V_2U_2^{-1}\\
V_1U_1^{-1} & ~~ 0
\end{array}\right).
\end{gather}
Using that $V_\ns U^{-1}=U^{-1}V_\ns$ by (\ref{uvrelation3}) and (\ref{vnsu2})  we derive 
$$\cC^2=V_\ns U^{-1}V_\ns U^{-1}=V_\ns^2 U^{-2}= I_{\Hh_\ns}.$$
Further, since $V_\ns $ is antiunitary and $U^{-1}$ is unitary, $\cC$ is antilinear and we have
$\|\cC x\|=\|V_\ns U^{-1}x\|=\|U^{-1}x\|=\|x\|$ for $x\in \Hh_\ns$, so $\cC$ is isometric. Hence $\cC$ is a conjugation on the Hilbert space $\Hh_\ns$ such that the operators $V_\ns$, $U$,  $\cC$ pairwise commute and
\begin{align}\label{vnsnsuc}
V_\ns= U\cC=\cC U.
\end{align}
This proves the assertion for $V_\ns$  in the case $T=V_T$, that is, for the special case of an antiunitary operator.\smallskip

Now we show the assertion for $T_\ns$  in the case of a general antinormal operator $T$. Recall from Theorem  \ref{charnpolar} that the positive self-adjoint operator
 $|T|$ commutes with $V_T$ and so with $V_T^2$, because $T$ is antinormal. Hence $|T|$ commutes with the spectral projections of $V_T^2$, that is, $$E_{V_T^2}(M)\, |T|\subseteq|T| E_{V_T^2}(M)$$ for each Borel subset $M$ of $ \dT$. Therefore, the closed subspaces $\Hh_\sa =E_{V_T^2}(\{ 1\})\Hh$,  
$\Hh_\ns=E_{V_T^2}(\dT \backslash \{1\})\Hh$,   $E_{V_T^2}(C_+)\Hh$,\, and $E_{V_T^2}(C_-)\Hh$ reduce the self-adjoint operator $|T|$.

Also, the closed subspace $\cG := E_{V_T^2}(\{-1\})\Hh$ reduces the operator $|T|$. Let $S$ denote the inverse of the operator $(I+|T|)\lceil \cG$ acting on the Hilbert space $\cG$.
Since $V_T$ commutes with $|T|$,  $J:=V_T\lceil \cG$ commutes with the  bounded self-adjoint operator $S$ on $\cG$.  Therefore,  Lemma \ref{galindo} applies and  the  correponding subspace $\cR$ is invariant under $S$. Hence the closed  subspaces $\cK_1$ and $\cK_2$ of $\cG$ defined by (\ref{defk1k2}) are invariant under $S$. This implies that $\cK_1$ and $\cK_2$ reduce the operator $|T|$. Hence $\Hh_1=E_{V_T^2}(C_+)\Hh\oplus \cK_1$ and $\Hh_2=E_{V_T^2}(C_-)\Hh\oplus \cK_2$ reduce $|T|$. From the preceding we conclude that 
$$
|T|\lceil \Hh_\ns= |T|\lceil \Hh_1\oplus | T|\lceil \Hh_2~~~\textrm{on}~~ \Hh_\ns=\Hh_1\oplus \Hh_2.
$$
Let us abbreviate $|T|_\ns:=|T|\lceil \Hh_\ns, |T|_1:=|T|\lceil \Hh_1, |T|_2:= | T|\lceil \Hh_2.$ Then
\begin{gather} 
|T|_\ns= \left(
\begin{array}{ll}
|T|_1 & ~0\\
~0 & |T|_2
\end{array}\right).
\end{gather}
Since $|T|$ commutes with the spectral projections $E_{V_T^2}(\cdot)$, it also commutes with the spectral integrals $U_1'$ and $U_2'$. Obviously, $|T|$ commutes with ${\rm i}\, I_{\cK_1}$ and $-{\rm i}\, I_{\cK_2}$, so it  commutes with the operators $U_1$ and $U_2$. Therefore, the block matrix $|T|_\ns$  commutes with the unitary operator block matrix $U$ defined by (\ref{defiU}), that is, $U|T|_\ns \subseteq |T|_\ns U.$ 
Hence $U|T|_\ns U^*\subseteq |T|_\ns.$ Since $U|T|_\ns U^*$ is self-adjoint, we get
$U|T|_\ns U^*=|T|_\ns$. Then $|T|_\ns =U^*|T|_\ns U$.  Thus, 
\begin{align}\label{comutns}
U|T|_\ns = |T|_\ns U~~ \textrm{and}~~ U^*|T|_\ns = |T|_\ns U^*.
\end{align}
Define $N_1:=U_1|T|_1$,~  $N_2:=U_2|T|_2$, and
\begin{gather} 
\cN:= \left(
\begin{array}{ll}
N_1 & ~0\\
~0 & N_2
\end{array}\right) =U|T|_\ns =|T|_\ns U,
\end{gather}
where the last equality follows from  (\ref{comutns}). Since $T=V_T|T|=|T|V_T$  and $V_\ns=U\cC =\cC U$ by (\ref{vnsnsuc}),  we conclude that 
$T_\ns= |T|_\ns V_\ns =|T|_\ns U \cC =\cN \cC$ and $T_\ns =V_\ns |T|_\ns =\cC U |T|_\ns =\cC \cN.$

Now we show that the linear operators $N_1$ and $N_2$ are normal. Using  (\ref{comutns}) we derive
\begin{align}\label{n1n2normal}\cN^*\cN=|T|_\ns U^*U|T|_\ns=|T|_\ns^2,~~
\cN \cN^*=|T|_\ns U U^*|T|_\ns =|T|_\ns^2.
\end{align}
Hence the diagonal operator block matrix $\cN$ is normal and so are their diagonal entries $N_1$ and $N_2$.

Next we prove that $N_1$ is supported by the set $K$.  Since $N_1=U_1|T|_1$ and the unitary operator $U_1$  and the positive self-adjoint operator $|T|_1$ commute, using the product property of  the spectral measure $E_{N_1}$ of the normal operator $N_1$ we obtain
$$E_{N_1}(\dC\backslash K)=E_{N_1}((\dT\backslash K_{\dT})\times [0,+\infty))=E_{U_1}(\dT\backslash K_\dT)E_{|T|_1}([0,+\infty)).$$
Since $U_1$ is supported by $K_\dT$ as shown above,  $E_{U_1}(\dT\backslash K_\dT)=0$ and hence $E_{N_1}(\dC\backslash K)=0$, that is, $N_1$ is supported by $K$.

The formulas (\ref{matrixtbt2}) and (\ref{matrixtbt3}) follow by straightforward computations. Note that (\ref{n1n2normal}) is just  the first formula of   (\ref{matrixtbt3}). 
All assertions about $T_\ns$ are proved. \smallskip

Now we turn to the case of $T_\sa$ which is much simpler. Recall that $ \Hh_\sa =E_{V_T^2}(\{ 1\})\Hh$ by (\ref{defhsahnsa}). From (\ref{EVA2}) it follows that $V_T$ leaves $\Hh_\sa$ invariant.  Since $V_T^2 x=x$ for $x\in \Hh_\sa$ and $V_T$ is antiunitary, $C_0:=V_T\lceil \Hh_\sa$ is a conjugation and $S:=|T|\lceil \Hh_\sa$ is a positive self-adjoint operator on $\Hh_\sa$. Since $T=V_T |T|=|T|V_T$, we have $T_\sa=SC_0=C_0S,$ see also Theorem \ref{charselfadjoint}.

This completes the proof of Theorem \ref{normalform}.
\bibliographystyle{plain}

\end{document}